\documentclass[12pt]{iopart}
\usepackage{amssymb}
\usepackage{amsthm}
\usepackage{graphicx}
\usepackage{graphics,epsfig,color}

\renewcommand{\S}{\mathbb{S}}
\newcommand{\nn}{{\mathbf{n},\mathbf{p}}}
\newcommand{\p}{\mathbf{p}}
\newcommand{\n}{\mathbf{n}}
\newcommand{\0}{\mathbf{0}}
\newcommand{\G}[1]{G^{(#1)}_\nn}
\newcommand{\E}[2]{E^{(#1)}_{#2,\nn}}
\newcommand{\D}[1]{D_{#1}}
\newcommand{\B}{B_1(\0)}

\newcommand{\N}{{\mathbf N}}

\newcommand{\R}{{\mathbf R}}

\newcommand{\C}{{\mathbf C}}

\newtheorem{lemma}{Lemma}[section]
\newtheorem{thm}[lemma]{Theorem}

\begin{document}
\bibliographystyle{plain}

\title[Cyclidic coordinate system]{Separation of variables in an asymmetric cyclidic coordinate system}

\author{H S Cohl${}^{1}$ and H Volkmer${}^2$}

\address{$^1$Applied and Computational Mathematics Division, 
National Institute of Standards and Technology, Gaithersburg, MD
20899-8910, USA}
\address{$^2$Department of Mathematical Sciences, University of Wisconsin--Milwaukee,
P.~O. Box 413, Milwaukee, WI 53201, USA}
%\ead{howard.cohl@nist.gov
%\ead{volkmer@uwm.edu}
\begin{abstract}
A global analysis is presented of solutions for Laplace's equation
on three-dimensional Euclidean space in one of the 
most general orthogonal asymmetric confocal cyclidic 
coordinate systems which admit solutions
through separation of variables.  We refer to this coordinate system
as five-cyclide coordinates since the coordinate surfaces are given
by two cyclides of genus zero which represent the inversion at the unit
sphere of each other, a cyclide of genus one, and
two disconnected cyclides of genus zero.  This coordinate system 
is obtained by stereographic projection of sphero-conal
coordinates on four-dimensional Euclidean space.  The harmonics in this
coordinate system are given by products of solutions of
second-order Fuchsian ordinary differential equations with 
five elementary singularities.  The Dirichlet problem for the 
global harmonics in this coordinate system is solved using 
multiparameter spectral theory in the regions bounded by 
the asymmetric confocal cyclidic coordinate surfaces.
\end{abstract}

\pacs{41.20.Cv, 02.30.Em, 02.30.Hq, 02.30.Jr}
\ams{35J05, 42C15, 34L05, 35J15, 34B30}
% Keywords: Separation of variables; Laplace's equation; Orthogonal curvilinear coordinate systems; Multiparameter spectral theory
\maketitle

\section{Introduction}
In 1894 Maxime B\^ocher's book ``Ueber die Reihenentwicklungen der
Potentialtheorie'' was published
\cite{Bocher}.
It took its origin from lectures given by Felix Klein in G\"ottingen
(see for instance \cite{Klein81,Klein91}).
In B\^{o}cher's book, the author gives a list
of 17 inequivalent coordinate systems in
three dimensions in which the Laplace equation admits separated solutions of
the form
\begin{equation}\label{1:sep}
 U(x,y,z)=R(x,y,z) w_1(s_1)w_2(s_2)w_3(s_3),
\end{equation}
where the modulation factor $R(x,y,z)$
\cite[p.~519]{MorseFesh} is a known and fixed function,
and $s_1,s_2,s_3$ are curvilinear
coordinates of $x,y,z$.
The functions $w_1,w_2,w_3$ are solutions of second order ordinary
differential equations.
The symmetry group of Laplace's equation is the conformal group and
equivalence between various separable coordinate systems is established
by the existence of a conformal transformation which maps one separable
coordinate system to another.

In general, the coordinate surfaces (called confocal cyclides) are
given by the zero sets of polynomials in $x,y,z$ of degree at most $4$
which can be broken up into several different subclasses.  For instance,
eleven of these coordinate systems have coordinate surfaces which are
given by confocal quadrics \cite[Systems 1--11 on p.~164]{Miller},
nine are rotationally-invariant \cite[Systems 2,5--8 on p.~164 and 14--17
on p.~210]{Miller}, four are cylindrical \cite[Systems 1--4]{Miller},
and five which are the most general, are of the asymmetric type namely,
confocal ellipsoidal, paraboloidal, sphero-conal, and two cyclidic
coordinate systems \cite[Systems 9,10,11 on p.~164
and Systems 12,13 on p.~210]{Miller}).  B\^ocher \cite{Bocher} showed
how to solve the Dirichlet problem for harmonic functions on
regions bounded by such confocal cyclides.  However, it is stated repeatedly
in B\^ocher's book that the presentation lacked convergence proofs, for
instance, this is mentioned in the preface written by Felix Klein.

It is the purpose of this paper to supply the missing proofs for one of the
asymmetric cyclidic coordinate systems which is listed as number 12 in Miller's
list \cite[page 210]{Miller}
(see also Boyer, Kalnins \& Miller (1976) \cite[Table 2]{BoyKalMil}
and Boyer, Kalnins \& Miller (1978) \cite{BoyKalMil78} for a more general setting).
For lack of a better name we call it {\it 5-cyclide coordinates}. This
asymmetric orthogonal curvilinear coordinate system has coordinates
$s_i\in\R$ $(i=1,2,3)$ with $s_i$ in $(a_0,a_1)$, $(a_1,a_2)$ or $(a_2,a_3)$,
respectively where $a_0<a_1<a_2<a_3$ are given numbers.
This coordinate system is described by coordinate
surfaces $s_i=const$ which are five compact cyclides. 
The surfaces
$s_1=const$ for $s_1\in(a_0,a_1)$ are two cyclides of 
genus zero representing inversions at the unit sphere 
of each other.
The surface $s_2=const$ for $s_2\in(a_1,a_2)$ represents
a ring cyclide of genus one and the surfaces $s_3=const$ for
$s_3\in(a_2,a_3)$ represent two disconnected cyclides of genus zero with reflection symmetry about the $x,y$-plane.
The asymptotic behavior of this coordinate
system as the size of these compact cyclides increases without limit
is {\it 6-sphere coordinates} (see Moon \& Spencer (1961)
\cite[p.~122]{MoonSpencer}), the inversion of Cartesian coordinates.

In our notation the coordinate surfaces of this system are given by
the variety
\begin{equation}\label{1:surfaces}
\frac{(x^2+y^2+z^2-1)^2}{s-a_0}+ \frac{4x^2}{s-a_1}+\frac{4y^2}{s-a_2}
+ \frac{4z^2}{s-a_3}=0,
\end{equation}
where $s=s_i$ is either in $(a_0,a_1)$, $(a_1,a_2)$ or $(a_2,a_3)$
respectively.
\begin{figure}[h]
\begin{center}
\includegraphics[height=74.5mm]{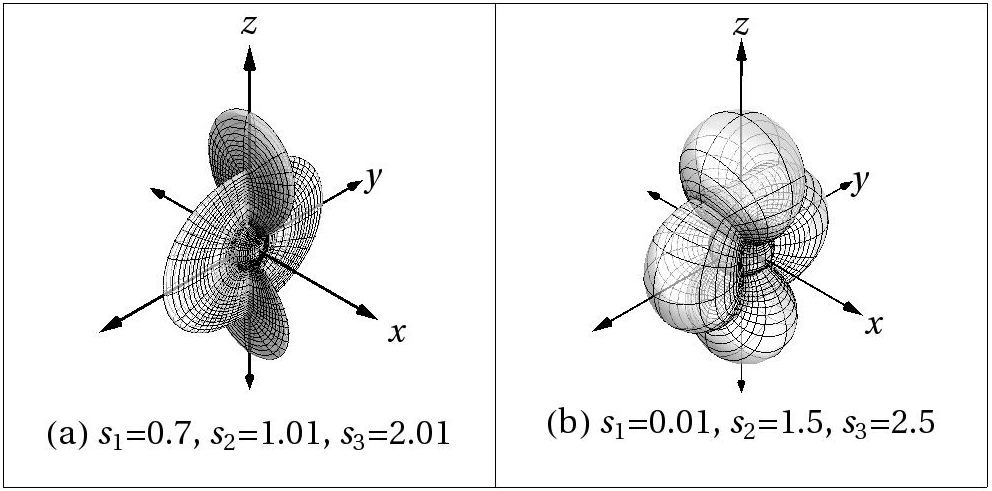}
\end{center}
\caption{Surfaces $s_{1,2,3}=const$ for $a_i=i,$ where
only the component of the cyclide $s_1=const$ inside the ball $x^2+y^2+z^2<1$ is shown.
\label{6:fig1}
}
\end{figure}
See Figure \ref{6:fig1}(a,b) for a graphical illustration of these
triply-orthogonal coordinate surfaces, where we have selected
one of the confocal cyclides for $s_1=const$.
This is a very general coordinate system containing the
parameters $a_0,a_1,a_2,a_3$ which generates many other coordinate
systems by limiting processes.
Since the book by B\^ocher is quite old and uses very geometrical methods,
we will present our results independently of B\^ocher's book.
We supply convergence proofs based on general multiparameter spectral
theory \cite{AtkMing11,Volkmerbook} which was created with such
applications in mind.  As far as we know this general theory has never
before been applied to the Dirichlet problems considered by B\^ocher.

We start with the observation that 5-cyclide coordinates are the stereographic
image of sphero-conal coordinates in four dimensions (or, expressed in another
way, of ellipsoidal coordinates on the hypersphere $\S^3$). We take the
sphero-conal coordinate system as known but we present the needed facts in
Section \ref{sphero}.  The well-known stereographic projection is dealt with
in Section \ref{stereo} which also explains the appearance of the factor $R$
in (\ref{1:sep}).  The 5-cyclide coordinate system is introduced in
Section \ref{coord}.  The solution of the Dirichlet problem on regions
bounded by surfaces (\ref{1:surfaces}) with $s\in(a_1,a_2)$ is presented
in Section \ref{Dirichlet1}. The preceding Section \ref{SL1} provides the
needed convergence proofs based on multiparameter spectral theory.
The remaining sections treat the Dirichlet problem on regions bounded by
the surfaces (\ref{1:surfaces}) when $s\in(a_1,a_2)$ (ring cyclides) and
$s\in(a_2,a_3)$.

\section{Sphero-conal coordinates on $\R^{k+1}$}\label{sphero}

Let $k\in\N$. In order to introduce sphero-conal coordinates on $\R^{k+1}$,
fix real numbers
\begin{equation}\label{2:a}
a_0<a_1<a_2<\dots <a_k.
\end{equation}
Let
$(x_0,x_1,\dots,x_k)$ be in the positive cone of
$\R^{k+1}$
\begin{equation}
x_0>0,\ldots,x_k>0.
\label{poscone}
\end{equation}
Its sphero-conal coordinates $r, s_1,\dots,s_k$ are
determined in the intervals
\begin{equation}\label{2:cube}
 r>0,\; a_{i-1}<s_i<a_i,\quad i=1,\dots,k
\end{equation}
by the equations
\begin{equation}\label{2:sphero1}
 r^2=\sum_{j=0}^k x_j^2
\end{equation}
and
\begin{equation}\label{2:sphero2}
 \sum_{j=0}^k \frac{x_j^2}{s_i-a_j}=0\quad\mbox{for\ } 
i=1,\dots,k.
\end{equation}
The latter equation determines $s_1,s_2,\dots,s_k$ as the zeros of a
polynomial of degree $k$ with coefficients which are polynomials
in $x_0^2,\dots,x_k^2$.

In this way we obtain a bijective (real-)analytic map from the positive
cone in $\R^{k+1}$ to the set of points $(r,s_1,\dots,s_k)$
satisfying (\ref{2:cube}).  The inverse map is found by solving a linear
system. It is also analytic, and it is given by
\begin{equation}\label{2:inverse}
  x_j^2=r^2 \frac{\prod_{i=1}^k(s_i-a_j)}{\prod_{j\ne i=0}^k (a_i-a_j)}.
\end{equation}
Sphero-conal coordinates are orthogonal, and
its scale factors (metric coefficients) are given by $H_r=1$, and
\begin{equation}\label{2:scaling}
 H_{s_i}^2=\frac14\sum_{j=0}^k \frac{x_j^2}{(s_i-a_j)^2}=-\frac14 r^2
 \frac{\prod_{i \ne j=1}^k (s_i-s_j)}
 {\prod_{j=0}^k (s_i-a_j)},\quad i=1,2,\dots,k .
\end{equation}

Consider the Laplace equation
\begin{equation}\label{2:pde}
\Delta U=\sum_{i=0}^k \frac{\partial^2 U}{\partial x_i^2} =0
\end{equation}
for a function $U(x_0,x_1,\dots,x_k)$.
Using (\ref{2:scaling}) we transform this equation
to sphero-conal coordinates, and then we apply the method of separation
of variables \cite{SchmidtWolf}
\begin{equation}\label{2:sepsol}
 U(x_0,x_1,\dots,x_k)=w_0(r)w_1(s_1)w_2(s_2)\dots w_k(s_k).
\end{equation}
For the variable $r$ we obtain the Euler equation
\begin{equation}\label{2:Euler}
 w_0''+\frac{k}{r} w_0'+ \frac{4\lambda_0}{r^2}w_0=0
\end{equation}
while for each of the variables $s_1,s_2,\dots,s_k$  we obtain the
Fuchsian equation
\begin{equation}\label{2:Fuchs}
\prod_{j=0}^k (s-a_j)\left[ w''+\frac12\sum_{j=0}^k \frac{1}
{s-a_j} w'\right] +\left[\sum_{i=0}^{k-1}\lambda_i s^{k-1-i}\right]w =0 .
\end{equation}
More precisely, if $\lambda_0,\ldots,\lambda_{k-1}$
are any given numbers (separation constants),
and if $w_0(r)$, $r>0$, solves (\ref{2:Euler})
and $w_i(s_i)$, $a_{i-1}<s_i<a_i$, solve (\ref{2:Fuchs})
for each $i=1,\dots,k$, then $U$ defined by (\ref{2:sepsol})
solves (\ref{2:pde}) in the positive cone of $\R^{k+1}$
(\ref{poscone}).

Equation (\ref{2:Fuchs}) has only regular points except for $k+2$ regular
singular points at $s=a_0,a_1,\dots,a_k$ and $s=\infty$.  The exponents at
each finite singularity $s=a_j$ are $0$ and $\frac12$. Therefore, for each
choice of parameters $\lambda_0,\dots,\lambda_{k-1}$,
there is a nontrivial analytic solution at $s=a_j$ and another one of the form
$w(s)=(s-a_j)^{1/2} v(s)$, where $v$ is analytic at $a_j$.
If $\nu,\mu$ denote the exponents at $s=\infty$ then
\begin{equation}\label{2:infinity}
\mu\nu=\lambda_0,\quad \mu+\nu=\frac{k-1}{2}.
\end{equation}
The polynomial $\sum_{i=0}^{k-1}\lambda_i s^{k-1-i}$ appearing
in (\ref{2:Fuchs}) is known as van Vleck polynomial.  If $k=1$
then (\ref{2:Fuchs}) is the hypergeometric differential equation (up to
a linear substitution). If $k=2$ then (\ref{2:Fuchs}) is the Heun equation.
We will use this equation for $k=3$.  According to Miller (1977)
\cite[p.~209]{Miller} (see also \cite[p.~71]{BoyKalMil})
in reference to the $k=3$ case, ``Very little is
known about the solutions.''

\section{Stereographic projection}\label{stereo}
We consider the stereographic projection
$P:\S^3\setminus\{(1,0,0,0)\}\to\R^3$ given by
\[
P(x_0,x_1,x_2,x_3)=\frac{1}{1-x_0}(x_1,x_2,x_3).
\]
The inverse map is
\[
P^{-1}(x,y,z)=\frac{1}{x^2+y^2+z^2+1}(x^2+y^2+z^2-1,2x,2y,2z).
\]
We extend $P^{-1}$ to a bijective map
\[Q:(0,\infty)\times \R^3\to \R^4\setminus\{(x_0,0,0,0): x_0\ge 0\}\]
by defining
\[
Q(r,x,y,z):=r P^{-1}(x,y,z).
\]
If we set $(x_0,x_1,x_2,x_3)=Q(r,x,y,z)$ we may consider $r,x,y,z$ as
curvilinear coordinates on $\R^4$ with Cartesian coordinates $x_0,x_1,x_2,x_3$.
We note that $x_0^2+x_1^2+x_2^2+x_3^2=r^2$ so $r$ is just the distance
between $(x_0,x_1,x_2,x_3)$ and the origin.  Moreover, $(x,y,z)$ is the
stereographic projection of the point $(x_0/r,x_1/r,x_2/r,x_3/r)\in\S^3$.
It is easy to check that the coordinate system is orthogonal and scale factors
are
\[
h_r=1,\quad h_x=h_y=h_z=2rh,\quad \mbox{where\ } 
h:=\frac{1}{x^2+y^2+z^2+1}.
\]
Let $U(x_0,x_1,x_2,x_3)=V(r,x,y,z)$. Then
\begin{equation}\label{3:eq1}
\Delta U = \frac{1}{8r^3 h^3}\left((2rhV_x)_x+(2rhV_y)_y
+(2rhV_z)_z+(8r^3h^3 V_r)_r\right).
\end{equation}
Suppose that $U$ is homogeneous of degree $\alpha$:
\[ U(tx_0,tx_1,tx_2,tx_3)=t^\alpha U(x_0,x_1,x_2,x_3),\quad t>0.\]
Then $V$ can be written in the form
\[
V(r,x,y,z)=r^\alpha w(x,y,z),
\]
and (\ref{3:eq1}) implies
\begin{equation}\label{3:eq2}
 \Delta U= \frac{r^{\alpha-2}}{4h^3}
 \left((hw_x)_x+(hw_y)_y+(hw_z)_z+4\alpha(\alpha+2)h^3w\right).
\end{equation}
We now introduce the function
\[ u(x,y,z)=w(x,y,z)(x^2+y^2+z^2+1)^{-1/2}  .\]
Then a direct calculation changes (\ref{3:eq2}) to
\begin{equation}\label{3:eq3}
\Delta U=\frac{r^{\alpha-2}}{4h^{5/2}}\left(u_{xx}+u_{yy}
+u_{zz}+(3+4\alpha(\alpha+2))h^2 u\right).
\end{equation}
If $3+4\alpha(\alpha+2)=0$ then $U$ is harmonic if and only if $u$ is harmonic.
Noting that $3+4\alpha(\alpha+2)=(2\alpha+1)(2\alpha+3)$,
we obtain the following theorem.

\begin{thm}\label{3:t1}
Let $D$ be an open subset of $\S^3$ not containing $(1,0,0,0)$, let
$E=\{(rx_0,rx_1,rx_2,rx_3): r>0, (x_0,x_1,x_2,x_3)\in D\}$, and let
$F=P(D)$ be the stereographic image of $D$.
Let the function $U:E\to\R$ be homogeneous of degree $-\frac12$ or
$-\frac32$, and let $w:F\to\R$ satisfy $U=w\circ P$ on $D$.
Then  $U$ is harmonic on $E$ if and only if $w(x,y,z)(x^2+y^2+z^2+1)^{-1/2}$
is harmonic on~$F$.
\end{thm}

\section{Five-cyclide coordinate system on $\R^3$}\label{coord}
We introduce sphero-conal coordinates
\[ r>0,\quad a_0<s_1<a_1<s_2<a_2<s_3<a_3 ,\]
on $\R^4$ as explained in Section \ref{sphero} with $k=3$.
Then $s_1,s_2,s_3$ form a coordinate system for the intersection of
the hypersphere $\S^3$ with the positive cone in $\R^4$.  Using the
stereographic projection $P$ from Section~\ref{stereo} we project these
coordinates to $\R^3$. We obtain a coordinate system for the set
\begin{equation}\label{4:set}
 T=\{(x,y,z): x,y,z>0,\, x^2+y^2+z^2>1 \} .
\end{equation}
Explicitly,
\begin{equation}\label{4:system12a}
 x= \frac{x_1}{1-x_0},\quad  y=\frac{x_2}{1-x_0},
 \quad z=\frac{x_3}{1-x_0} ,
\end{equation}
where
\begin{equation}\label{4:system12b}
x_j^2=\frac{\prod_{i=1}^3(s_i-a_j)}{\prod_{j\ne i=0}^3 (a_i-a_j)},\quad j=0,1,2,3.
\end{equation}
Conversely, the coordinates $s_1,s_2,s_3$ of a point $(x,y,z)\in T$ are the solutions of
\begin{equation}\label{4:surface}
\frac{(x^2+y^2+z^2-1)^2}{s-a_0}+ \frac{4x^2}{s-a_1}+\frac{4y^2}{s-a_2}+ \frac{4z^2}{s-a_3}=0 .
\end{equation}
Since sphero-conal coordinates are orthogonal and the stereographic
projection preserves angles, 5-cyclide coordinates are orthogonal, too.
This is the twelfth coordinate system in Miller (1977) \cite[page 210]{Miller}.
Miller uses a slightly different notation:
$a_0=0$, $a_1=1$, $a_2=b$, $a_3=a$, and
$s_1=\rho$, $s_2=\nu$, $s_3=\mu$. Also, $x,z$ are interchanged.

In order to calculate the scale factors for the 5-cyclide coordinate system
we proceed as follows.  We start with
\[
\frac{\partial x}{\partial s_i}=\frac{1}{1-x_0}\frac{\partial x_1}
{\partial s_i}+\frac{x_1}{(1-x_0)^2}\frac{\partial x_0}{\partial s_i},
\]
and similar formulas for the derivatives of $y$ and $z$.  Then using
\[
x_0^2+x_1^2+x_2^2+x_3^2=1, \quad x_0 \frac{\partial x_0}{\partial s_i}
+x_1\frac{\partial x_1}{\partial s_i}+x_2 \frac{\partial x_2}{\partial s_i}
+x_3\frac{\partial x_3}{\partial s_i}=0
\]
a short calculation gives
\[
\frac{\partial x}{\partial s_i}\frac{\partial x}{\partial s_j}
+\frac{\partial y}{\partial s_i}\frac{\partial y}{\partial s_j}
+\frac{\partial z}{\partial s_i}\frac{\partial z}{\partial s_j}
=\frac{1}{(1-x_0)^2} \sum_{\ell=0}^3 \frac{\partial x_\ell}{\partial s_i}
\frac{\partial x_\ell}{\partial s_j}.
\]
This confirms that 5-cyclide coordinates are orthogonal and
from (\ref{2:scaling}) we obtain the squares of their scale factors
\begin{equation}\label{4:scaling1}
h_i^2=\frac{1}{16}\left(\frac{(\rho^2-1)^2}{(s_i-a_0)^2}
+\frac{4x^2}{(s_i-a_1)^2}+\frac{4y^2}{(s_i-a_2)^2}+\frac{4z^2}
{(s_i-a_3)^2} \right) ,
\end{equation}
where $\rho^2=x^2+y^2+z^2$, or, equivalently,
\begin{eqnarray}
h_1^2&=&\frac1{16}(\rho^2+1)^2\frac{(s_3-s_1)(s_2-s_1)}
{(s_1-a_0)(a_1-s_1)(a_2-s_1)(a_3-s_1)},\label{4:h1}\\
h_2^2&=&\frac1{16}(\rho^2+1)^2\frac{(s_2-s_1)(s_3-s_2)}
{(s_2-a_0)(s_2-a_1)(a_2-s_2)(a_3-s_2)},\label{4:h2}\\
h_3^2&=&\frac1{16}(\rho^2+1)^2\frac{(s_3-s_1)(s_3-s_2)}
{(s_3-a_0)(s_3-a_1)(s_3-a_2)(a_3-s_3)} \label{4:h3}.
\end{eqnarray}

We find harmonic functions by separation of variables in 5-cyclide coordinates as follows.

\begin{thm}\label{4:t1}
Let $w_1:(a_0,a_1)\to\C$, $w_2:(a_1,a_2)\to\C$, $w_3:(a_2,a_3)\to \C$
be solutions of the Fuchsian equation
\begin{equation}\label{4:Fuchs}
\prod_{j=0}^3 (s-a_j)\left[w''+\frac12 \sum_{j=0}^3
\frac{1}{s-a_j} w'\right] +\left(\frac{3}{16}s^2+\lambda_1s+\lambda_2\right)w =0,
\end{equation}
where $\lambda_1,\lambda_2$ are given (separation) constants.
Then the function
\begin{equation}\label{4:sepsol}
 u(x,y,z)=(x^2+y^2+z^2+1)^{-1/2}w_1(s_1)w_2(s_2)w_3(s_3)
\end{equation}
is a harmonic function on the set (\ref{4:set}).
\end{thm}
\begin{proof}
Using sphero-conal coordinates $r,s_1,s_2,s_3$ on $\R^4$, we define a
function $U$ in the positive cone of $\R^4$ by
\[
U(x_0,x_1,x_2,x_3)=r^{-1/2}w_1(s_1)w_2(s_2)w_3(s_3).
\]
The function $r^{-1/2}$ is a solution of (\ref{2:Euler}) when
$k=3$, $\lambda_0=\frac{3}{16}$.  The results from Section \ref{sphero}
imply that $U$ is harmonic, and, of course, $U$ is homogeneous of
degree $-\frac12$.  The function $w$ defined on the set (\ref{4:set}) by
$U=w\circ P$ is given in 5-cyclide coordinates by
\[
w(x,y,z)=w_1(s_1)w_2(s_2)w_3(s_3).
\]
Therefore, Theorem \ref{3:t1} gives the statement of the theorem.
\end{proof}

Equation (\ref{4:Fuchs}) has five regular singularities at $s=a_0,a_1,a_2,a_3,\infty$. The exponents at the finite singularities are $0$ and $\frac12$.
Using (\ref{2:infinity}), we find that the exponents at infinity are $\frac14$ and $\frac34$. So all five singularities are elementary in the sense of Ince \cite{Ince}.
Equation (\ref{4:Fuchs}) is one of the standard equations in the classification of Ince \cite[page 500]{Ince}.

We define the 5-cyclide coordinates $s_1,s_2,s_3$ for an arbitrary
point $(x,y,z)\in\R^3$ as the zeros $s_1\le s_2\le s_3$ of the cubic equation
\begin{equation}\label{4:cubic}
\prod_{j=0}^ 3(s-a_j)\left[\frac{(x^2+y^2+z^2-1)^2}{s-a_0}
+\frac{4x^2}{s-a_1}+\frac{4y^2}{s-a_2}+ \frac{4z^2}{s-a_3}\right]=0 .
\end{equation}
For example, $s_j(0,0,0)=a_j$ for $j=1,2,3$.
Each function $s_j:\R^3\to [a_{j-1},a_j]$ is continuous.
We observe that, in general, there are 16 different points in $\R^3$ which have the
same coordinates $s_1,s_2,s_3$. If $(x,y,z)$ is one of these points
the other ones are obtained by applying the group generated by inversion at $\S^2$
\begin{equation}\label{4:inversion}
\sigma_0(x,y,z)=\rho^{-2}(x,y,z)
\end{equation}
and reflections at the coordinate planes
\begin{equation}\label{4:reflections}
\sigma_1(x,y,z)=(-x,y,z),\,\sigma_2(x,y,z)=(x,-y,z),\, \sigma_3(x,y,z)=(x,y,-z).
\end{equation}
It is of interest to determine the sets where $s_j=a_{j-1}$ or $s_j=a_j$.
We obtain
\begin{eqnarray}
s_1&=&a_0 \mbox{\ iff\ } x^2+y^2+z^2=1,\\
s_1&=&a_1\mbox{\ iff\ } x=0 \mbox{\ and\  } 
\frac{(\rho^2-1)^2}{a_1-a_0}+\frac{4y^2}{a_1-a_2}
+\frac{4z^2}{a_1-a_3}\ge 0, \\
s_2&=&a_1\mbox{\ iff\ }x=0 \mbox{\ and\ } 
\frac{(\rho^2-1)^2}{a_1-a_0}+ 
\frac{4y^2}{a_1-a_2}+\frac{4z^2}{a_1-a_3}\le 0, \\
s_2&=&a_2\mbox{\ iff\ } y=0 \mbox{\ and\ } 
\frac{(\rho^2-1)^2}{a_2-a_0}+
\frac{4x^2}{a_2-a_1}+\frac{4z^2}{a_2-a_3}\ge 0,\label{4:set4} \\
s_3&=&a_2\mbox{\ iff\ } y=0 \mbox{\ and\ } 
\frac{(\rho^2-1)^2}{a_2-a_0}
+\frac{4x^2}{a_2-a_1}+\frac{4z^2}{a_2-a_3}\le 0, \\
s_3&=&a_3 \mbox{\ iff\ } z=0.
\end{eqnarray}

We define the sets (consisting each of two closed curves)
\begin{eqnarray}\label{4:A1}
A_1&:=&\{(x,y,z)\in\R^3: s_1=s_2=a_1\}\\
&=&\{(x,y,z): x=0, \frac{(\rho^2-1)^2}{a_1-a_0}
+\frac{4y^2}{a_1-a_2}+\frac{4z^2}{a_1-a_3}=0\}, \nonumber
\end{eqnarray}
see Figure \ref{4:fig1}, and
\begin{eqnarray}\label{4:A2}
A_2&:=&\{(x,y,z)\in\R^3: s_2=s_3=a_2\}\\
&=&\{(x,y,z): y=0, \frac{(\rho^2-1)^2}{a_2-a_0}
+\frac{4x^2}{a_2-a_1}+\frac{4z^2}{a_2-a_3}=0\}, \nonumber
\end{eqnarray}
see Figure \ref{4:fig2}.
\begin{figure}[h]
\centering\includegraphics[width=3in,height=3in]{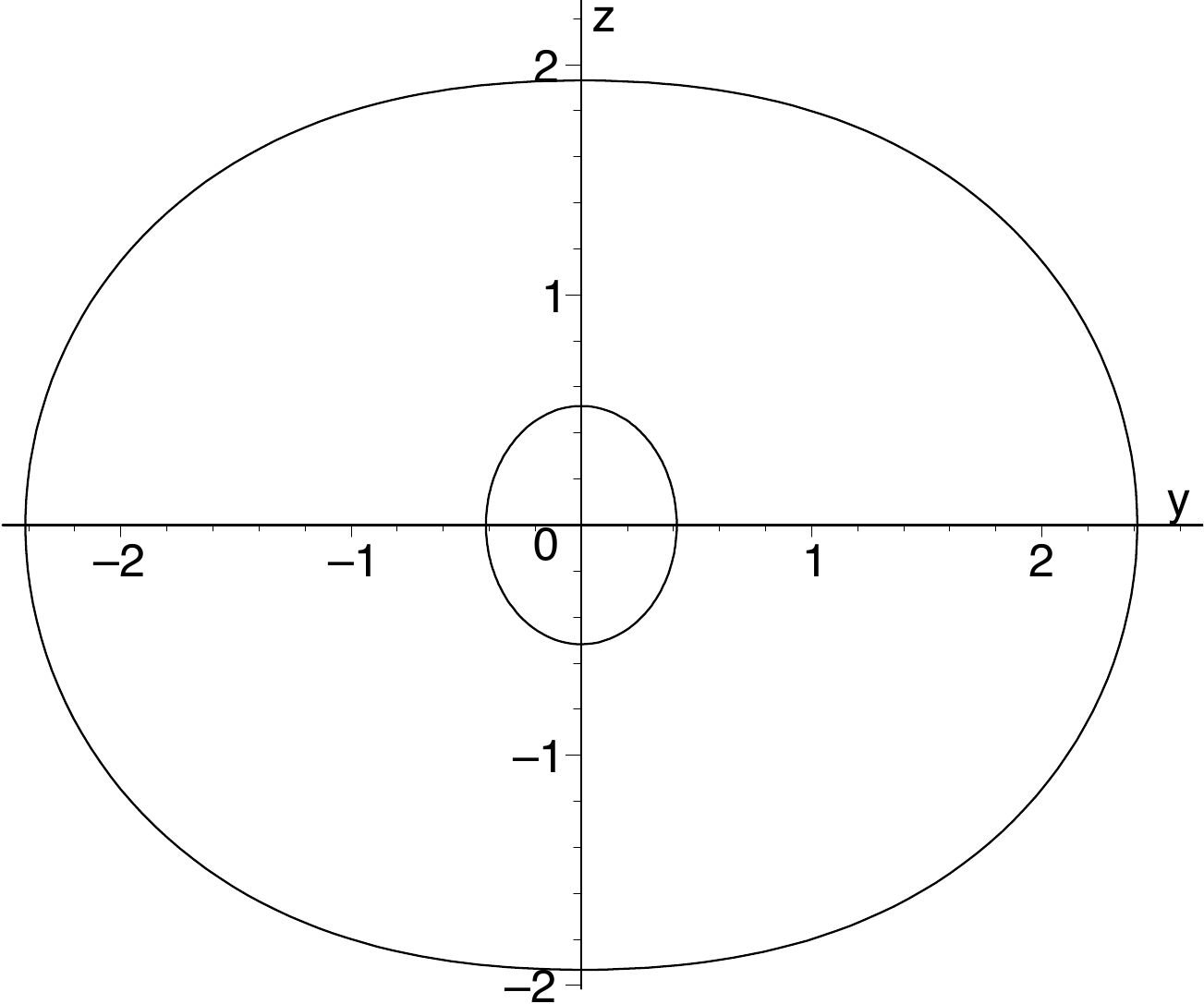}
\caption{The set $A_1$ in the $y,z$-plane for $a_i=i$.\label{4:fig1}}
\end{figure}
\begin{figure}[h]
\centering\includegraphics[width=2in,height=3.5in]{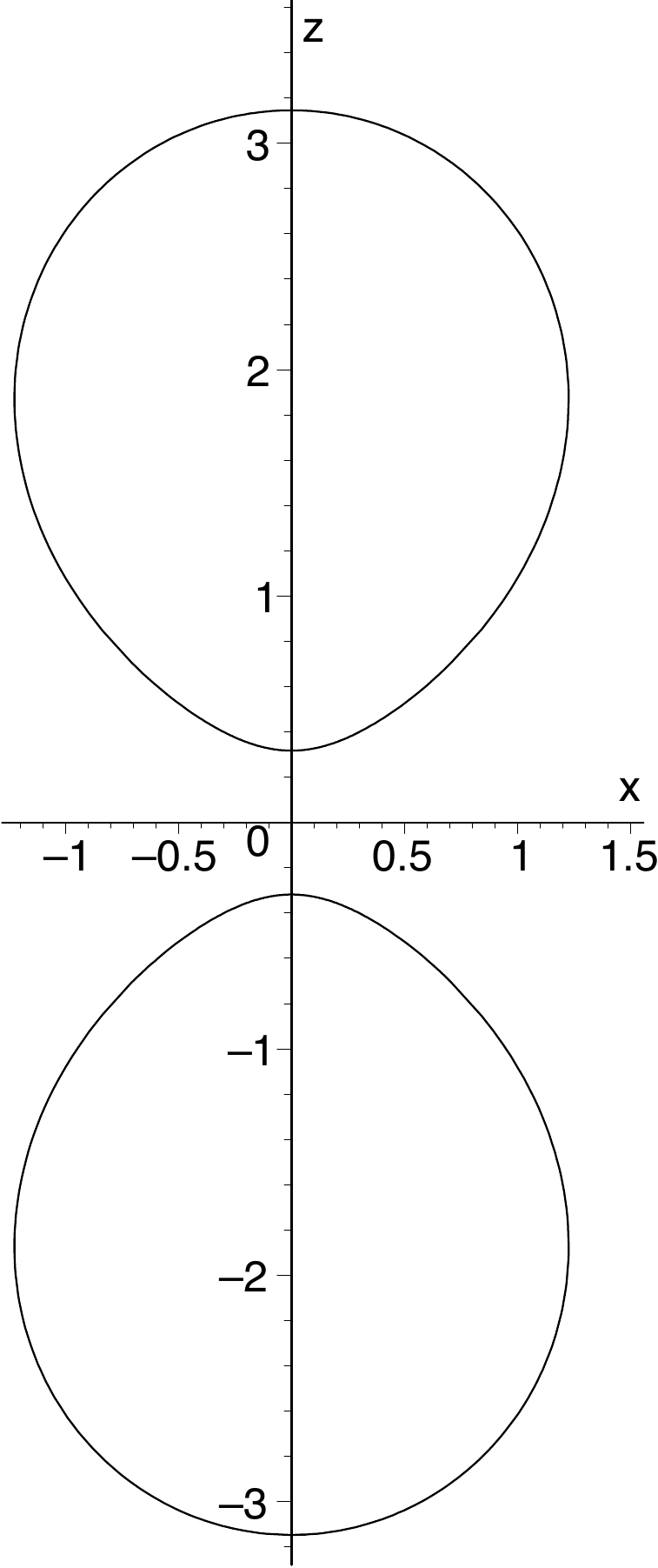}
\caption{The set $A_2$ in the $x,z$-plane for $a_i=i$.\label{4:fig2}}
\end{figure}
Clearly, $s_j$ is analytic at all points $(x,y,z)$ at which $s_j$ is a
simple zero of the cubic equation (\ref{4:cubic}).  Therefore, $s_1$ is
analytic on $\R^3\setminus A_1$, $s_2$ is analytic on
$\R^3\setminus(A_1\cup A_2)$, and $s_3$ is analytic on $\R^3\setminus A_2$.

We may use (\ref{4:sepsol}) to define $u(x,y,z)$ for all $(x,y,z)\in\R^3$.
Since the solutions $w_1,w_2,w_3$ of (\ref{4:Fuchs}) have limits at the
end points of their intervals of definition
(because the exponents are $0$ and $\frac12$ there),
we see that $u$ is a continuous functions on $\R^3$. The
function $(x^2+y^2+z^2+1)^{1/2}u(x,y,z)$ is invariant under
$\sigma_i$, $i=0,1,2,3$. In general, $u$ is harmonic only away from
the coordinate planes and the unit sphere.  In fact, we observe
that $u$ is a bounded function which converges to $0$ at infinity,
so, by Liouville's theorem, $u$ cannot be harmonic on all of $\R^3$
unless it is identically zero.

\section{First two-parameter Sturm-Liouville problem}\label{SL1}
We consider equation (\ref{4:Fuchs}) on the intervals $(a_1,a_2)$
and $(a_2,a_3)$ and write it in formally self-adjoint form.  Setting
\begin{equation}\label{5:p}
 \omega(s):=\left|(s-a_0)(s-a_1)(s-a_2)(s-a_3)\right|^{1/2},
\end{equation}
we obtain two Sturm-Liouville  equations involving two parameters
\begin{eqnarray}
&&(\omega(s_2)w_2')'+\frac{1}{\omega(s_2)}
\left(\frac3{16}s_2^2+\lambda_1 s_2+\lambda_2\right) w_2 =0,
\quad a_1<s_2<a_2,\label{5:SL1}\\
&&(\omega(s_3)w_3')'-\frac{1}{\omega(s_3)}\left(\frac3{16}s_3^2+\lambda_1
s_3+\lambda_2\right) w_3 =0,\quad a_2<s_3<a_3.\label{5:SL2}
\end{eqnarray}
In (\ref{5:SL1}) $w_2$ is a function of $s_2$ and the derivatives are
taken with respect to $s_2$.  In (\ref{5:SL2}) $w_3$ is a function
of $s_3$ and the derivatives are taken with respect to $s_3$.
We simplify the equations by substituting $t_j=\Omega(s_j)$,
$u_j(t_j)=w_j(s_j)$, where $\Omega(s)$ is the elliptic integral
(see for instance \cite{Klein91})
\begin{equation}\label{5:st}
 \Omega(s):=\int_{a_0}^s \frac{d\sigma}{\omega(\sigma)} .
\end{equation}
This is an increasing  absolutely continuous
function $\Omega:[a_0,a_3]\to [0,b_3]$, where $b_j:=\Omega(a_j)$.
Let $\phi:[0,b_3]\to [a_0,a_3]$ be the inverse function of $\Omega$.
Then (\ref{5:SL1}), (\ref{5:SL2}) become
\begin{eqnarray}
 u_2''+\left(\frac3{16} \{\phi(t_2)\}^2+\lambda_1 \phi(t_2)
 +\lambda_2\right) u_2 &=&0,\quad b_1\le t_2\le b_2 ,\label{5:SL3}\\
 u_3''-\left(\frac3{16} \{\phi(t_3)\}^2+\lambda_1 \phi(t_3)
 +\lambda_2\right) u_3 &=&0,\quad b_2\le t_3\le b_3 .\label{5:SL4}
\end{eqnarray}

We add the boundary conditions
\begin{equation}\label{5:bc}
 u_2'(b_1)=u_2'(b_2)=u_3'(b_2)=u_3'(b_3)=0 .
\end{equation}

Differential equations (\ref{5:SL3}), (\ref{5:SL4}) together with boundary
conditions (\ref{5:bc}) pose a two-parameter Sturm-Liouville
eigenvalue problem. For the theory of such multiparameter problems we
refer to the books \cite{AtkMing11,Volkmerbook} and the references
therein.  A pair $(\lambda_1,\lambda_2)$ is called an {\it eigenvalue} if
there exist (nontrivial) {\it eigenfunctions} $u_2(t_2)$ and $u_3(t_3)$
which satisfy (\ref{5:SL3}), (\ref{5:SL4}), (\ref{5:bc}).  The two-parameter
problem is {\it right-definite} in the sense that
\[
\left|
\begin{array}{cc}
\phi(t_2) & 1 \\ -\phi(t_3) & -1
\end{array}
\right|
=\phi(t_3)-\phi(t_2)>0 \quad
\mbox{\ for\ } b_1<t_2<b_2<t_3<b_3.
\]
However, this determinant is not positive on the closed
rectangle $[b_1,b_2]\times[b_2,b_3]$. This lack of {\it uniform
right-definiteness} make some proofs in this section a little longer
than they would be otherwise.

We have the following Klein oscillation theorem; see \cite[Theorem 5.5.1]{AtkMing11}.

\begin{thm}\label{5:t1}
For every $\n=(n_2,n_3)\in\N_0^2$, there exists a uniquely determined
eigenvalue $(\lambda_{1,\n}, \lambda_{2,\n})\in\R^2$ admitting an eigenfunction
$u_2$ with exactly $n_2$ zeros in $(b_1,b_2)$ and an eigenfunction $u_3$
with exactly $n_3$ zeros in $(b_2,b_3)$.
\end{thm}

We state a result on the distribution of eigenvalues; compare
with \cite[Chapter 8]{AtkMing11}.

\begin{thm}\label{5:t2}
There are positive constants $A_1,A_2,A_3,A_4$ such that, for all $\n\in\N_0^2$,
\begin{eqnarray}\label{5:ineq}
& -A_1(n_2^2+n_3^2+1)\le \lambda_{1,\n}\le -A_2(n_2^2+n_3^2)+A_3,& \label{5:est1}\\
&|\lambda_{2,\n}|\le A_4 (n_2^2+n_3^2+1).&\label{5:est2}
\end{eqnarray}
\end{thm}
\begin{proof}
If a differential equation $u''+q(t)u=0$ with continuous $q:[a,b]\to\R$
admits a solution $u$ satisfying $u'(a)=u'(b)=0$ and having exactly $m$ zeros
in $(a,b)$, then there is $t\in(a,b)$ such that $q(t)=\frac{\pi^2 m^2}{(b-a)^2}$.
This is shown by comparing with the eigenvalue problem
$u''+\lambda u=0$, $u'(a)=u'(b)=0$.
Applying this fact,
we find $t_2\in(b_1,b_2)$ and $t_3\in(b_2,b_3)$ such that
\begin{eqnarray}\label{5:compare1}
\frac{3}{16}\{\phi(t_2)\}^2 +\lambda_1 \phi(t_2)
+\lambda_2 &=&\frac{\pi^2 n_2^2}{(b_2-b_1)^2}, \\
\frac{3}{16}\{\phi(t_3)\}^2 +\lambda_1 \phi(t_3)
+\lambda_2 &=&-\frac{\pi^2 n_3^2}{(b_3-b_2)^2} \label{5:compare2},
\end{eqnarray}
where we abbreviated $\lambda_j=\lambda_{j,\n}$.
By subtracting (\ref{5:compare1}) from (\ref{5:compare2}), we
obtain
\[
\frac{3}{16}\left(\{\phi(t_3)\}^2-\{\phi(t_2)\}^2\right)
+\lambda_1(\phi(t_3)-\phi(t_2)) = -\frac{\pi^2 n_2^2}{(b_2-b_1)^2}
-\frac{\pi^2 n_3^2}{(b_3-b_2)^2}\le 0 .
\]
Dividing by $\phi(t_3)-\phi(t_2)$ and using
$0<\phi(t_3)-\phi(t_2)\le a_3-a_1$, we obtain the second
inequality in (\ref{5:est1}).

To prove the first inequality in (\ref{5:est1}), suppose that
$\lambda_1<-\frac{3}{8} a_3$. Then the van Vleck polynomial
\begin{equation}\label{5:Q}
 Q(s):=\frac{3}{16} s^2+\lambda_1 s+\lambda_2
\end{equation}
satisfies $Q'(s)=\frac{3}{8}s+\lambda_1<0$ for $s\le a_3$.
Let $c\in(b_1,b_2)$ be determined by $\phi(c)=\frac12(a_1+a_2)$.
If $Q(a_2)\ge 0$ then, for  $t\in[b_1,c]$,
\[
Q(\phi(t))\ge Q(\frac12(a_1+a_2))\ge \frac12(a_2-a_1)
\left(-\lambda_1-\frac{3}{8}a_3\right) .
\]
By Sturm's comparison theorem applied to equation (\ref{5:SL3}), we get
\[
(c-b_1)^2 (a_2-a_1)\left(-\lambda_1-\frac{3}{8}a_3\right) \le 4\pi^2(n_2+1)^2,
\]
which gives the desired inequality.
If $Q(a_2)<0$ we argue similarly working with (\ref{5:SL4}) instead.

Finally, (\ref{5:est2}) follows from (\ref{5:est1}) and (\ref{5:compare1}).
\end{proof}

Let $u_{2,\n}$ and $u_{3,\n}$ denote real-valued eigenfunctions corresponding
to the eigenvalue $(\lambda_{1,\n},\lambda_{2,\n})$.  It is
known \cite[section 3.5]{AtkMing11} (and easy to prove) that the system
of products $u_{2,\n}(t_2)u_{3,\n}(t_3)$, $\n\in\N_0^2$, is orthogonal in
the Hilbert space $H_1$ consisting of measurable functions
$f:(b_1,b_2)\times (b_2,b_3)\to\C$ satisfying
\[
\int_{b_2}^{b_3}\int_{b_1}^{b_2} (\phi(t_3)-\phi(t_2))
\left|f(t_2,t_3)\right|^2\,dt_2\,dt_3<\infty\]
with inner product
\[
\int_{b_2}^{b_3}\int_{b_1}^{b_2} (\phi(t_3)-\phi(t_2)) f(t_2,t_3)
\overline{g(t_2,t_3)}\,dt_2\,dt_3.
\]
We normalize the eigenfunctions so that
\begin{equation}\label{5:norm1}
\int_{b_2}^{b_3}\int_{b_1}^{b_2} (\phi(t_3)-\phi(t_2))
\left\{u_{2,\n}(t_2)\right\}^2\left\{u_{3,\n}(t_3)\right\}^2\,dt_2\,dt_3 =1 .
\end{equation}

We have the following completeness theorem; see \cite[Theorem 6.8.3]{Volkmerbook}.

\begin{thm}\label{5:t3}
The double sequence of functions
\[ u_{2,\n}(t_2)u_{3,\n}(t_3),\quad \n\in \N_0^2 ,\]
forms an orthonormal basis in the Hilbert space $H_1$.
\end{thm}

The normalization (\ref{5:norm1}) leads to a bound on the values of eigenfunctions.

\begin{thm}\label{5:t4}
There is a constant $B>0$ such that, for all $\n\in\N_0^2$ and
all $t_2\in[b_1,b_2]$, $t_3\in[b_2,b_3]$,
\[
|u_{2,\n}(t_2)u_{3,\n}(t_3)|\le B(n_2^2+n_3^2+1).
\]
\end{thm}
\begin{proof}
We abbreviate $u_j=u_{j,\n}$, $\lambda_j=\lambda_{j,\n}$.
Condition (\ref{5:norm1}) is a normalization for the product
$u_2(t_2)u_3(t_3)$ but not for each factor separately, so we may assume
that, additionally,
\begin{equation}\label{5:norm2}
 \int_{b_1}^{b_2} \left\{u_2(t_2)\right\}^2\,dt =1 .
 \end{equation}
Now (\ref{5:norm1}), (\ref{5:norm2}) imply that
\begin{equation}\label{5:norm3}
\int_{b_2}^{b_3} (\phi(t_3)-\phi(b_2))\left\{u_3(t_3)\right\}^2 \,dt_3 \le 1 .
\end{equation}

We multiply equations (\ref{5:SL3}), (\ref{5:SL4}) by $u_2$ and $u_3$,
respectively, and integrate by parts to obtain
\begin{eqnarray}
\int_{b_1}^{b_2} u_2'^2 & = &\frac3{16}\int_{b_1}^{b_2} \phi^2 u_2^2
+\lambda_1 \int_{b_1}^{b_2} \phi u_2^2+\lambda_2
\int_{b_1}^{b_2} u_2^2 ,\label{5:eq1}\\
\int_{b_2}^{b_3} u_3'^2 &=& -\frac3{16}\int_{b_2}^{b_3} \phi^2 u_3^2
-\lambda_1 \int_{b_2}^{b_3} \phi u_3^2-\lambda_2
\int_{b_2}^{b_3} u_3^2 .\label{5:eq2}
\end{eqnarray}
It follows from (\ref{5:norm2}), (\ref{5:eq1}) and Theorem \ref{5:t2} that
there is a constant $B_1>0$ such that, for all $\n\in\N_0^2$,
\begin{equation}\label{5:est3}
\int_{b_1}^{b_2} u_2'^2 \le B_1(n_2^2+n_3^2+1) .
\end{equation}
Unfortunately, we cannot argue the same way for $u_3$ because we do not
have an upper bound for $\int_{b_2}^{b_3} u_3^2$.  Instead, we multiply
(\ref{5:eq1}) by $\int u_3^2$ and (\ref{5:eq2}) by $\int u_2^2$ and add
the equations.  Then, noting (\ref{5:norm1}), we find
\[
\int_{b_1}^{b_2} u_2'^2 \int_{b_2}^{b_3} u_3^2+\int_{b_1}^{b_2} u_2^2 \int_{b_2}^{b_3} u_3'^2\\
\le -\lambda_1 +\frac38\max_{t\in[b_1,b_3]} |\phi(t)| .
\]
Using Theorem \ref{5:t2} and (\ref{5:norm2}), we find a constant
$B_2>0$ such that, for all $\n\in\N_0^2$,
\begin{equation}\label{5:ineq3}
  \int_{b_2}^{b_3} u_3'^2 \le B_2(n_2^2+n_3^2+1).
\end{equation}
We apply the following Lemma \ref{5:l1} (noting (\ref{5:norm2}),
(\ref{5:norm3}), (\ref{5:est3}), (\ref{5:ineq3}))
and obtain the desired result.
\end{proof}

\begin{lemma}\label{5:l1}
Let $u:[a,b]\to\R$ be a continuously differentiable function, and
let $a\le c<d\le b$.
Then, for all $t\in[a,b]$,
\[
(d-c)\left|u(t)\right|^2 \le 2 \int_c^d \left|u(r)\right|^2\,dr
+ 2(b-a)(d-c) \int_a^b \left|u'(r)\right|^2\,dr .\]
\end{lemma}
\begin{proof}
For $s,t\in[a,b]$ we have
\[
|u(t)-u(s)|=\left|\int_s^t u'(r)\,dr\right|\le |t-s|^{1/2}
\left(\int_a^b \left|u'(r)\right|^2\, dr\right)^{1/2} .\]
This implies
\[
|u(t)|^2\le 2|u(s)|^2 + 2|t-s| \int_a^b \left|u'(r)\right|^2\, dr.
\]
We integrate from $s=c$ to $s=d$ and obtain the desired inequality.
\end{proof}

Let $u_{1,\n}$ be the solution of
\begin{equation}\label{5:ode}
u_1''-\left(\frac3{16} \{\phi(t_1)\}^2+\lambda_{1,\n} \phi(t_1)
+\lambda_{2,\n}\right) u_1 =0,\quad b_0\le t_1\le b_1,
\end{equation}
determined by the initial conditions
\[
u_1(b_1)=1,\quad u_1'(b_1)=0 .
\]
The following estimate on $u_{1,\n}$ will be useful in Section \ref{Dirichlet1}.

\begin{thm}\label{5:t5}
We have $u_{1,\n}(t_1)>0$ for all $t_1\in[b_0,b_1]$.
If $0=b_0\le c_1<c_2<b_1$, then there are constants $C>0$ and $0<r<1$ such
that, for all $\n\in\N_0^2$ and $t_1\in[c_2,b_1]$,
\[
\frac{u_{1,\n}(t_1)}{u_{1,\n}(c_1)} \le C r^{n_2+n_3}.
\]
\end{thm}
\begin{proof}
We abbreviate $u_1=u_{1,\n}$ and $\lambda_j=\lambda_{j,\n}$.
By definition, $u_1$ satisfies the differential equation
\[ u_1''=Q(\phi(t_1))u_1 ,\quad t_1\in[b_0,b_1] ,\]
where $Q$ is given by (\ref{5:Q}).
According to (\ref{5:compare1}), (\ref{5:compare2}), there are
$s_2\in(a_1,a_2)$ and $s_3\in(a_2,a_3)$ such that
\[
Q(s_2)=\frac{\pi^2 n_2^2}{(b_2-b_1)^2},\quad Q(s_3)=-\frac{\pi^2 n_3^2}
{(b_3-b_2)^2}.
\]
If $s\le s_2$ then $Q(s)\ge L(s)$, where $L(s)$ is the linear function
with $L(s_j)=Q(s_j)$, $j=2,3$.
It follows that $Q(s)\ge 0$ for $s\in[a_0,a_1]$ and
\begin{equation}\label{5:ineq4}
 Q(\phi(t_1)) \ge C_1(n_2+n_3)^2 \quad \mbox{\ for\  } t_1\in [b_0,c_2],
\end{equation}
where $C_1$ is a positive constant independent of $\n$.
We now apply the following Lemma \ref{5:l2}
to complete the proof. Note that the interval $[a,b]$ in the lemma
is $[c_1,b_1]$ but with the end points interchanged.
\end{proof}

\begin{lemma}\label{5:l2}
Let $u:[a,b]\to\R$ be a solution of the differential equation
\[ u''(t)=q(t)u(t),\quad t\in[a,b], \]
determined by the initial conditions $u(a)=1$, $u'(a)=0$,
where $q:[a,b]\to\R$ is a continuous function. Suppose that $q(t)\ge 0$ on $[a,b]$
and $q(t)\ge \lambda^2$ on $[c,b]$ for some $\lambda>0$ and $c\in[a,b)$. Then
$u(t)>0$ for all $t\in[a,b]$, and
\[
\frac{u(b)}{u(t)}\ge \frac12 e^{\lambda(b-c)}
\ \mbox{for\ all} \ t\in[0,c].
\]
\end{lemma}
\begin{proof}
Since $q(t)\ge 0$, $u(t)>0$ and $u'(t)\ge 0$ for $t\in[a,b]$.
The function $z=u'/u$ satisfies the Riccati equation
\[ z'+z^2=q(t), \]
and the initial condition $z(a)=0$.
It follows that
\[ z(t)\ge \lambda \tanh(\lambda(t-c))\quad \mbox{\ for\ } t\in[c,b]. \]
Integrating from $t=c$ to $t=b$ gives
\[ \ln \frac{u(b)}{u(c)} \ge \ln \cosh \lambda(b-c)\ge
\ln \frac12 e^{\lambda(b-c)} \]
which yields the claim since $u$ is nondecreasing.
\end{proof}

We now introduce a systematic notation for our eigenvalues and  eigenfunctions.
First of all, we note that the results of this section remain valid for
other sets of boundary conditions.
We will need eight sets of boundary
conditions labeled by $\p=(p_1,p_2,p_3)\in\{0,1\}^3$. These
boundary conditions are
\begin{equation}\label{5:gbc}
\begin{array}{llll}
u_2'(b_1)=0 & \ \mbox{\ if\ } p_1=0,&\quad  u_2(b_1)=0 &\ 
\mbox{\ if\ } p_1=1,\\
u_2'(b_2)=u_3'(b_2)=0 &\ \mbox{\ if\  }p_2=0, &\quad  u_2(b_2)=u_3(b_2)=0
&\ \mbox{\ if\ } p_2=1,\\
u_3'(b_3)=0 &\ \mbox{\ if\ } p_3=0,&\quad u_3(b_3)=0 &\ 
\mbox{\ if\ }p_3=1.
\end{array}
\end{equation}
The initial conditions for $u_1$ are
\begin{equation}
\label{5:ini}
u_1(b_1)=1, u_1'(b_1)=0\,\ \mbox{\ if\ } p_1=0,
\quad u_1(b_1)=0, u_1'(b_1)=1\,
\ \mbox{\ if\ }p_1=1.
\end{equation}
We denote the corresponding eigenvalues
by $(\lambda^{(1)}_{1,\nn},\lambda^{(1)}_{2,\nn})$.
For the notation of eigenfunctions we return to the $s_i$-variable
connected to $t_i$ by $t_i=\Omega(s_i)$.
The eigenfunctions will be denoted by $\E{1}{i}(s_i)=u_{i,\n}(t_i)$, $i=1,2,3$.
The superscript $(1)$ is used to distinguish from eigenvalues and
eigenfunctions introduced in Sections 7 and 9.
The subscript $\n=(n_2,n_3)$ indicates the number of zeros
of $\E{1}{2}(s_2)$, $\E{1}{3}(s_3)$ in $(a_1,a_2)$, $(a_2,a_3)$, respectively.
The subscript $\p$ indicates the boundary conditions used to determine
eigenvalues and eigenfunctions.  By using the letter $E$
for eigenfunctions we follow  B\^ocher \cite{Bocher}.
In our notation we suppressed the dependence of eigenvalues and
eigenfunctions on $a_0,a_1,a_2,a_3$.

Summarizing, for $i=1,2,3$, $\E{1}{i}$ is a solution of
(\ref{4:Fuchs}) on
$(a_{i-1},a_i)$ with $(\lambda_1,\lambda_2)
=(\lambda^{(1)}_{1,\nn},\lambda^{(2)}_{2,\nn})$.
The solution $\E{1}{1}(s_1)$ has exponent $\frac12 p_1$ at $a_1$ and
it has no zeros in $(a_0,a_1)$. The solution $\E{1}{2}(s_2)$
has exponent $\frac12 p_1$ at $a_1$, exponent $\frac12 p_2$ at $a_2$,
and its has $n_2$ zeros in $(a_1,a_2)$.  The solution $\E{1}{3}(s_3)$
has exponent $\frac12 p_2$ at $a_2$, exponent $\frac12 p_3$ at $a_3$,
and it has $n_3$ zeros in $(a_2,a_3)$.

\section{First Dirichlet problem}\label{Dirichlet1}

Consider the coordinate surface (\ref{4:surface}) for fixed $s=d_1\in(a_0,a_1)$.
See Figure \ref{s1const} for a graphical depiction of the shape of this surface.
\begin{figure}[h]
\begin{center}
\includegraphics[height=101.5mm]{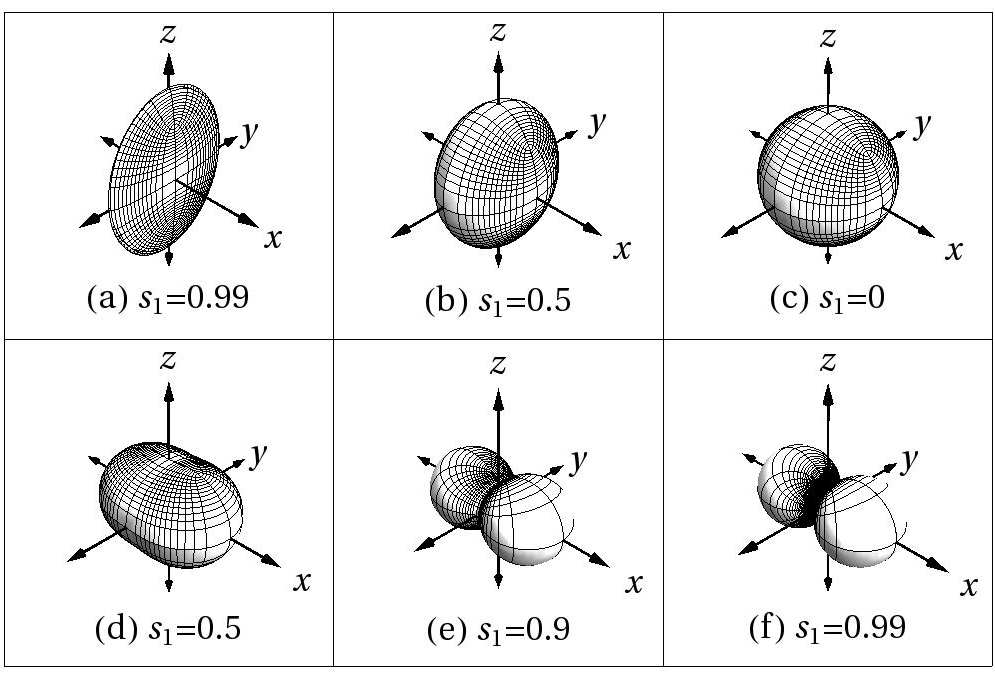}
\end{center}
\caption{Coordinate surfaces $s_1=const$ for $a_i=i$ with (a), (b)
inside $\B$, (d), (e), (f) outside $\B$, and (c) the unit sphere.
\label{s1const} }
\end{figure}
Let $(x',y',z')\in\S^2$. The ray $(x,y,z)=t(x',y',z')$, $t>0$, intersects
the surface if
\begin{equation}
\label{6:surface2}
\frac{(t^2-1)^2}{d_1-a_0}=ct^2 ,
\end{equation}
where
\[
c=\frac{4x'^2}{a_1-d_1}+\frac{4y'^2}{a_2-d_1}+\frac{4z'^2}{a_3-d_1}>0.
\]
Equation (\ref{6:surface2}) has two positive solutions $t=t_1,t_2$ such
that $0<t_1<1<t_2$ and $t_1t_2=1$.  Therefore, the coordinate surface
$s_1=d_1$ consists of two disjoint closed surfaces of genus $0$. One lies
inside the unit ball $\B$ centered at the origin and the other one is the image
of it under the inversion (\ref{4:inversion}).  Let $\D{1}$ be the region
interior to the first surface, that is,
\begin{equation}
\label{6:D1}
\D{1}= \{(x,y,z)\in \B: s_1>d_1\},
\end{equation}
or, equivalently,
\[
\D{1}=\{(x,y,z)\in \B: \frac{(\rho^2-1)^2}{d_1-a_0}
+\frac{4x^2}{d_1-a_1}+\frac{4y^2}{d_1-a_2}+ \frac{4z^2}{d_1-a_3}>0\}.
\]
We showed that $\D{1}$ is star-shaped with respect to the origin.
We now solve the Dirichlet problem for harmonic functions in $\D{1}$ by the
method of separation of variables.

Let $\p=(p_1,p_2,p_3)\in\{0,1\}^3$ and $\n=(n_2,n_3)\in\N_0^2$. Using the
functions $\E{1}{i}$ introduced in Section \ref{SL1} we define
the {\it internal 5-cyclidic harmonic of the first kind}
\begin{equation}\label{6:G}
\G{1}(x,y,z)=(x^2+y^2+z^2+1)^{-1/2}\E{1}{1}(s_1)\E{1}{2}(s_2)\E{1}{3}(s_3)
\end{equation}
for $x,y,z\in \B$ with $x,y,z\ge 0$. We extend this function to $\B$ as
a function of parity $\p$.  We call a function $f$ of parity $\p$ if
\begin{equation}
\label{6:parity}
f(\sigma_i(x,y,z))=(-1)^{p_i} f(x,y,z),\quad \mbox{\ for\ } i=1,2,3
\end{equation}
using the reflections (\ref{4:reflections}).

\begin{lemma}\label{6:l1}
The function $\G{1}$ is harmonic on $\B$.
\end{lemma}
\begin{proof}
By Theorem \ref{4:t1}, $\G{1}$ is harmonic on $\B$ away from the coordinate planes.
Therefore, it is enough to show that $\G{1}$ is analytic on $\B$.

Consider first $\p=(0,0,0)$. Then (\ref{6:G}) holds on $\B$.
Since $s_1\ne a_0$ on $\B$, $\G{1}$ is analytic on
$\B\setminus(A_1\cup A_2)$ as a composition of analytic functions.
In order to show that $\G{1}$ is also analytic at the points of
$A_1 \cup A_2$, one may refer to a classical result on ``singular curves'' of
harmonic functions; see Kellogg (1967) \cite[Theorem XIII, page 271]{Kellogg} but we
will argue more directly.  Since $A_1$ and $A_2$ are disjoint sets, it is
clear that $\E{1}{3}(s_3)$ is analytic at every point in $\B\cap A_1$.
In order to show that $\E{1}{1}(s_1)\E{1}{2}(s_2)$ is analytic at $(x',y',z')\in\B\cap A_1$, we argue as follows.
We may assume that there is an analytic function $w:(a_1,a_3)\to\R$ such
that $\E{1}{1}$ and $\E{1}{2}$ are restrictions of this function to
$(a_1,a_2)$ and $(a_2,a_3)$, respectively.  Now $(s_1-a_1)+(s_2-a_1)$ and $(s_1-a_1)(s_2-a_1)$ are analytic functions of $(x,y,z)$
in a neighborhood of $(x',y',z')$.
Lemma \ref{6:l2} implies that
$\E{1}{1}(s_1)\E{1}{2}(s_2)$ as a function of $(x,y,z)$ is analytic
at $(x',y',z')$. It follows that $\G{1}$ is analytic at
every point in $\B\cap A_1$.  In the same way, we show that $\G{1}$ is
analytic at every point in $\B \cap A_2$.

If $\p=(0,0,1)$ then we introduce the function
\[
\chi:=
\begin{array}{ll}
\sqrt{a_3-s_3} & \mbox{\ if\ } z\ge 0\\[0.1cm]
-\sqrt{a_3-s_3} & \mbox{otherwise}. 
\end{array}
\]
It follows from (\ref{4:system12a}), (\ref{4:system12b}) that $\chi$ is analytic on $\R^3\setminus A_2$.
Then
\[
\G{1}(x,y,z)=(x^2+y^2+z^2+1)^{-1/2}\E{1}{1}(s_1)\E{1}{2}(s_2)
\chi(x,y,z)w_3(s_3) \]
on $\B$, where $w_3$ is analytic at $s_3=a_3$.  We then argue as above.

The other parity vectors $\p$ are treated similarly.
\end{proof}

\begin{lemma}\label{6:l2}
Let $f:(B_\epsilon)^2\to\C$, $B_\epsilon=\{s\in\C: |s|<\epsilon\}$, be an analytic function which is symmetric: $f(s,t)=f(t,s)$.
Let $g,h:(B_\delta)^3\to B_\epsilon$ be functions such that $g+h$ and $gh$ are analytic. Then the function $f(g(x,y,z),h(x,y,z))$ is analytic on $(B_\delta)^3$.
\end{lemma}

Substituting $t_j=\Omega(s_j)$, $j=2,3$, the Hilbert space $H_1$ from
Section~\ref{SL1}
transforms to the Hilbert space $\tilde{H}_1$ consisting of measurable
functions $g:(a_1,a_2)\times (a_2,a_3)\to\C$ for which
\begin{equation}
\label{6:Hilbert}
 \|g\|^2:= \int_{a_2}^{a_3}\int_{a_1}^{a_2} \frac{s_3-s_2}
 {\omega(s_2)\omega(s_3)} |g(s_2,s_3)|^2\,ds_2\,ds_3 <\infty.
\end{equation}
By Theorem \ref{5:t3}, for $g\in \tilde{H}_1$ and fixed $\p$, we have the
Fourier expansion
\begin{equation}\label{6:Fourier}
g(s_2,s_3)\sim \sum_\n c_{\nn} \E{1}{2}(s_2)\E{1}{3}(s_3),
\end{equation}
where the Fourier coefficients are given by
\begin{equation}\label{6:c1}
 c_{\nn}= \int_{a_2}^{a_3}\int_{a_1}^{a_2} \frac{s_3-s_2}{\omega(s_2)\omega(s_3)} g(s_2,s_3) \E{1}{2}(s_2)\E{1}{3}(s_3)\,ds_2\,ds_3 .
\end{equation}

\begin{thm}\label{6:t1}
Consider the region $\D{1}$ defined by (\ref{6:D1}) for some fixed
$d_1\in(a_0,a_1)$.  Let $e$ be a function defined on its boundary
$\partial \D{1}$ of parity $\p\in\{0,1\}^3$, and let $g(s_2,s_3)$
be the representation of
\begin{equation}\label{6:ef}
 f(x,y,z):=(x^2+y^2+z^2+1)^{1/2}e(x,y,z)
\end{equation}
in 5-cyclide coordinates for $(x,y,z)\in\partial D_1$ with $x,y,z>0$.
Suppose $g\in\tilde H_1$ and expand $g$ in the series (\ref{6:Fourier}).
Then the function $u(x,y,z)$ given by
\begin{equation}\label{6:expansion}
 u(x,y,z) =\sum_\n \frac{c_{\nn}}{\E{1}{1}(d_1)} \G{1}(x,y,z)
\end{equation}
is harmonic in $\D{1}$ and assumes the values $e$ on the boundary of
$\D{1}$ in the $L^2$-sense explained below.
\end{thm}
\begin{proof}
Let $d_1<d<a_1$ and $s_1\in[d,a_1]$. Using Theorems \ref{5:t4}, \ref{5:t5}
we estimate
\[
\left|c_{\nn}\frac{\E{1}{1}(s_1)}{\E{1}{1}(d_1)}
\E{1}{2}(s_2)\E{1}{3}(s_3)\right|\le |c_{\nn}|
Cr^{n_2+n_3} B(n_2^2+n_3^2+1) ,\]
where the constants $B,C>0$ and $r\in(0,1)$ are independent
of $\n$ and $s_1\in[d,a_1]$, $s_2\in[a_1,a_2]$, $s_3\in[a_2,a_3]$.
Since $c_{\nn}$ is a bounded double sequence, this proves that the
series in (\ref{6:expansion}) is absolutely and uniformly convergent
on compact subsets of $\D{1}$. Consequently, by Lemma \ref{6:l1}, $u(x,y,z)$ is harmonic
in $\D{1}$.  If we consider $u$ for fixed $s_1\in(d_1,a_1)$ and compute
the norm $\|u-e\|$ in the Hilbert space $\tilde H_1$ by the Parseval
equality, we obtain
\[
\|u-e\|^2\le \sum_\n |c_{\nn}|^2
\left(1-\frac{\E{1}{1}(s_1)}{\E{1}{1}(d_1)}\right)^2 .
\]
It is easy to see that the right-hand side converges to $0$ as $s_1\to d_1$.
Taking into account that
$e$ and $u$ have the same parity, it follows that $u$ assumes the boundary
values $e$ in this $L^2$-sense.
\end{proof}

If $e$ is a function on $\partial \D{1}$ without parity, we write the
function $f$ from (\ref{6:ef}) as a sum of eight functions
\[
f=\sum_\p f_\p,
\]
where $f_\p$ is of parity $\p$. Then the solution of the corresponding
Dirichlet problem is given by
\begin{equation}\label{6:expansion2}
 u(x,y,z)=\sum_\nn \frac{c_{\nn}}{\E{1}{1}(d_1)} \G{1}(x,y,z),
\end{equation}
where
\begin{equation}\label{6:c2}
 c_\nn=\int_{a_2}^{a_3}\int_{a_1}^{a_2} \frac{s_3-s_2}
 {\omega(s_2)\omega(s_3)} g_\p(s_2,s_3) \E{1}{2}(s_2)\E{1}{3}(s_3)\,ds_2\,ds_3
\end{equation}
and $g_\p(s_2,s_3)$ is the representation of $f_\p$ in 5-cyclide coordinates.

We may write the coefficient $c_\nn$ as an integral over the
surface $\partial \D{1}$ itself. The surface element is
$dS = h_2 h_3 \,ds_2\,ds_3$ with the scale factors $h_2, h_3$ given
in (\ref{4:h2}), (\ref{4:h3}).
Using
\[
\frac{h_2h_3}{h_1}=\frac14(x^2+y^2+z^2+1)\frac{\omega(s_1)}
{\omega(s_2)\omega(s_3)} (s_3-s_2),
\]
we obtain from (\ref{6:c2})
\begin{equation}\label{6:c3}
c_\nn=\frac{1}{2\omega(d_1) \E{1}{1}(d_1)} \int_{\partial \D{1}}
\frac{e}{h_1} \G{1} \, dS ,
\end{equation}
where
\[
h_1^2=\frac{1}{16}\left(\frac{(x^2+y^2+z^2-1)^2}{(d_1-a_0)^2}
+\frac{4x^2}{(d_1-a_1)^2}+\frac{4y^2}{(d_1-a_2)^2}
+\frac{4z^2}{(d_1-a_3)^2} \right) .\]

\section{Second two-parameter Sturm-Liouville problem}\label{SL2}

We treat the two-parameter eigenvalue problem that appears when we wish to solve the Dirichlet problem in ring cyclides.
It is quite similar to the one considered in
Section \ref{SL1}, however, there are also some interesting
differences.  Consider equation (\ref{4:Fuchs}) on the intervals $(a_0,a_1)$
and $(a_2,a_3)$.  We obtain two Sturm-Liouville  equations involving two
parameters
\begin{eqnarray}
&&(\omega(s_1)w_1')'-\frac{1}{\omega(s_1)}\left(\frac3{16}s_1^2
+\lambda_1 s_1+\lambda_2\right) w_1 =0,\quad a_0<s_1<a_1,\label{7:SL1}\\
&&(\omega(s_3)w_3')'-\frac{1}{\omega(s_3)}\left(\frac3{16}s_3^2+\lambda_1
s_3+\lambda_2\right) w_3 =0,\quad a_2<s_3<a_3.\label{7:SL2}
\end{eqnarray}
We again simplify by substituting $t_j=\Omega(s_j)$, $u_j(t_j)=w_j(s_j)$.
Then (\ref{7:SL1}), (\ref{7:SL2}) become
\begin{eqnarray}
 u_1''-\left(\frac3{16} \{\phi(t_1)\}^2+\lambda_1 \phi(t_1)
 +\lambda_2\right) u_1 &=&0,\quad b_0\le t_1\le b_1 ,\label{7:SL3}\\
 u_3''-\left(\frac3{16} \{\phi(t_3)\}^2+\lambda_1 \phi(t_3)
 +\lambda_2\right) u_3 &=&0,\quad b_2\le t_3\le b_3 .\label{7:SL4}
\end{eqnarray}
We add boundary conditions
\begin{equation}\label{7:bc}
 u_1'(b_0)=u_1'(b_1)=u_3'(b_2)=u_3'(b_3)=0 .
\end{equation}
Differential equations (\ref{7:SL3}), (\ref{7:SL4}) together with boundary
conditions (\ref{7:bc}) pose a two-parameter Sturm-Liouville
eigenvalue problem.
In contrast to Section \ref{SL1}, we now have a uniformly
right-definite problem:
\[
-
\left|
\begin{array}{cc} 
\phi(t_1) & 1 \\ \phi(t_3) & 1 
\end{array}
\right|
=\phi(t_3)-\phi(t_1)\ge a_2-a_1>0 \quad \mbox{\ for\ }b_0\le t_1
\le b_1\le t_3 \le b_3.
\]

We again have Klein's oscillation theorem.

\begin{thm}\label{7:t1}
For every $\n=(n_1,n_3)\in\N_0^2$, there exists a uniquely determined
eigenvalue $(\lambda_{1,\n}, \lambda_{2,\n})\in\R^2$
admitting an eigenfunction
$u_1$ with exactly $n_1$ zeros in $(b_0,b_1)$ and an
eigenfunction $u_3$ with exactly $n_3$ zeros in $(b_2,b_3)$.
\end{thm}

We state a result on the distribution of eigenvalues.

\begin{thm}\label{7:t2}
There are constants $A_1,A_2,A_3>0$ such that, for all $\n\in\N_0^2$,
\begin{eqnarray}\label{7:ineq}
&-A_1(n_3^2+1)\le \lambda_{1,\n}\le A_2(n_1^2+1),& \label{7:est1}\\
&|\lambda_{2,\n}|\le A_3(n_1^2+n_3^2+1).&\label{7:est2}
\end{eqnarray}
\end{thm}
\begin{proof}
We abbreviate $\lambda_j=\lambda_{j,\n}$.
Arguing as in the proof of Theorem \ref{5:t2},
there are $t_1\in [b_0,b_1]$ and $t_3\in[b_2,b_3]$ such that
\begin{eqnarray}
\frac{3}{16}\{\phi(t_1)\}^2 +\lambda_1 \phi(t_1) +\lambda_2 &=&
-\frac{\pi^2 n_1^2}{(b_1-b_0)^2},\label{7:compare1} \\
\frac{3}{16}\{\phi(t_3)\}^2 +\lambda_1 \phi(t_3) +\lambda_2 &=&
-\frac{\pi^2 n_3^2}{(b_3-b_2)^2}.\label{7:compare2}
\end{eqnarray}
By subtracting (\ref{7:compare1}) from (\ref{7:compare2}), we obtain
\[
\frac{3}{16}\left(\{\phi(t_3)\}^2-\{\phi(t_1)\}^2\right)
+\lambda_1(\phi(t_3)-\phi(t_1)) = \frac{\pi^2 n_1^2}{(b_1-b_0)^2}
-\frac{\pi^2 n_3^2}{(b_3-b_2)^2}
\]
which implies (\ref{7:est1}).
Now (\ref{7:est2}) follows from (\ref{7:est1}) and (\ref{7:compare1}).
\end{proof}

Let $u_{1,\n}$ and $u_{3,\n}$ denote eigenfunctions corresponding to the
eigenvalue $(\lambda_{1,\n},\lambda_{2,\n})$.
The system of products $u_{1,\n}(t_1)u_{3,\n}(t_3)$, $\n\in\N_0^2$, is
orthogonal in the Hilbert space $H_2$ consisting of measurable functions
$f:(b_0,b_1)\times (b_2,b_3)\to\C$ satisfying
\[
\int_{b_2}^{b_3}\int_{b_0}^{b_1} (\phi(t_3)-\phi(t_1))\left|f(t_1,t_3)
\right|^2\,dt_1\,dt_3<\infty\]
with inner product
\[
\int_{b_2}^{b_3}\int_{b_0}^{b_1} (\phi(t_3)-\phi(t_1))
f(t_1,t_3)\overline{g(t_1,t_3)}\,dt_1\,dt_3.
\]
We normalize the eigenfunctions so that
\begin{equation}\label{7:norm1}
 \int_{b_2}^{b_3}\int_{b_0}^{b_1} (\phi(t_3)-\phi(t_1))
 \left\{u_{1,\n}(t_1)\right\}^2\left\{u_{3,\n}(t_3)\right\}^2\,dt_1\,dt_3 =1 .
\end{equation}

We have the following completeness theorem.

\begin{thm}\label{7:t3}
The double sequence of functions
\[ u_{1,\n}(t_1)u_{3,\n}(t_3),\quad \n\in \N_0^2 ,\]
forms an orthonormal basis in the Hilbert space $H_2$.
\end{thm}

The normalization (\ref{7:norm1}) leads to a bound on the values of eigenfunctions.
Since we have uniform right-definiteness, the proof is simpler than the
proof of Theorem \ref{5:t4}.

\begin{thm}\label{7:t4}
There is a constant $B>0$ such that, for all $\n\in\N_0^2$ and
all $t_1\in[b_0,b_1]$, $t_3\in[b_2,b_3]$,
\[
|u_{1,\n}(t_1)u_{3,\n}(t_3)|\le B(n_1^2+n_3^2+1).
\]
\end{thm}

Let $u_{2,\n}$ be the solution of
\begin{equation}\label{7:ode}
u_2''+\left(\frac3{16} \{\phi(t_2)\}^2+\lambda_{1,\n} \phi(t_2)
+\lambda_{2,\n}\right) u_2 =0,\quad b_1\le t_2\le b_2
\end{equation}
determined by initial conditions
\[
u_2(b_1)=1,\quad u_2'(b_1)=0 .
\]

\begin{thm}\label{7:t5}
We have $u_{2,\n}(t_2)>0$ for all $t_2\in[b_1,b_2]$. If
$b_1< c_1<c_2<b_2$, then there are constants $C>0$ and $0<r<1$ such that,
for all $\n\in\N_0^2$ and $t_2\in[b_1,c_1]$,
\[
\left|\frac{u_{2,\n}(t_2)}{u_{2,\n}(c_2)}\right| \le C r^{n_1+n_3}.
\]
\end{thm}
\begin{proof}
We abbreviate $u_2=u_{2,\n}$ and $\lambda_j=\lambda_{j,\n}$.
We write (\ref{7:ode}) in the form
\[ u_2''+Q(\phi(t_2))u_2=0 ,\quad t_2\in[b_1,b_2] ,\]
where $Q$ is given by (\ref{5:Q}).
According to (\ref{7:compare1}), (\ref{7:compare2}), there are
$s_1\in(a_0,a_1)$ and $s_3\in(a_2,a_3)$ such that
\[
Q(s_1)=-\frac{\pi^2 n_1^2}{(b_1-b_0)^2},\quad Q(s_3)=
-\frac{\pi^2 n_3^2}{(b_3-b_2)^2}.
\]
If $s\in [s_1,s_3]$, then $Q(s)\le L(s)$, where $L(s)$ is the
linear function with $L(s_j)=Q(s_j)$, $j=1,3$. It follows that
\begin{equation}
\label{7:ineq4}
Q(\phi(t_2)) \le - C (n_1+n_3)^2 \quad \mbox{\ for\ } t_2\in [c_1,c_2].
\end{equation}
Again, we apply Lemma \ref{5:l2} to complete the proof.
\end{proof}

The results of this section remain valid for other boundary conditions.
This time we will need sixteen sets of boundary
conditions labeled by $\p=(p_0,p_1,p_2,p_3)\in\{0,1\}^4$. These
boundary conditions are
\begin{equation}\label{7:gbc}
\begin{array}{llll}
u_1'(b_0)=0 &\mbox{\ if\  }p_0=0,&\quad u_1(b_0)=0 &
\mbox{\ if\ } p_0=1,\\
u_1'(b_1)=0 &\mbox{\ if\ } p_1=0,&\quad  u_1(b_1)=0 &
\mbox{\ if\ } p_1=1,\\
u_3'(b_2)=0 &\mbox{\ if\ } p_2=0,&\quad  u_3(b_2)=0 &
\mbox{\ if\ } p_2=1,\\
u_3'(b_3)=0 &\mbox{\ if\ } p_3=0,&\quad u_3(b_3)=0 &
\mbox{\ if\  }p_3=1.
\end{array}
\end{equation}
The initial conditions for $u_2$ are
\begin{equation}\label{7:ini}
u_2(b_1)=1, u_2'(b_1)=0\, \mbox{\ if\ } p_1=0,\quad u_2(b_1)=0, u_2'(b_1)=1\,
\mbox{\ if\ } p_1=1.
\end{equation}
We denote the corresponding eigenvalues by
$(\lambda^{(2)}_{1,\nn},\lambda^{(2)}_{2,\nn})$.
The eigenfunctions will be denoted
by $\E{2}{i}(s_i)=u_{i,\n}(t_i)$, $i=1,2,3$.

Summarizing, for $i=1,2,3$, $\E{2}{i}$ is a solution 
of (\ref{4:Fuchs})
on $(a_{i-1},a_i)$ with $(\lambda_1,\lambda_2)
=(\lambda^{(2)}_{1,\nn},\lambda^{(2)}_{2,\nn})$.
The solution $\E{2}{1}(s_1)$ has exponent $\frac12 p_0$ at $a_0$,
exponent $\frac12 p_1$ at $a_1$, and it has $n_1$ zeros in $(a_0,a_1)$.
The solution $\E{2}{2}(s_2)$ has exponent $\frac12 p_1$ at $a_1$, and its
has no zeros in $(a_1,a_2)$.  The solution $\E{2}{3}(s_3)$ has
exponent $\frac12 p_2$ at $a_2$, exponent $\frac12 p_3$ at $a_3$, and it
has $n_3$ zeros in $(a_2,a_3)$.

\section{Second Dirichlet problem}\label{Dirichlet2}

Consider the coordinate surface (\ref{4:surface}) for fixed $s=d_2\in(a_1,a_2)$.
See Figure \ref{s23const}(a,b,c) for a graphical depiction of the shape of this surface.
\begin{figure}[h]
\begin{center}
\includegraphics[height=100mm]{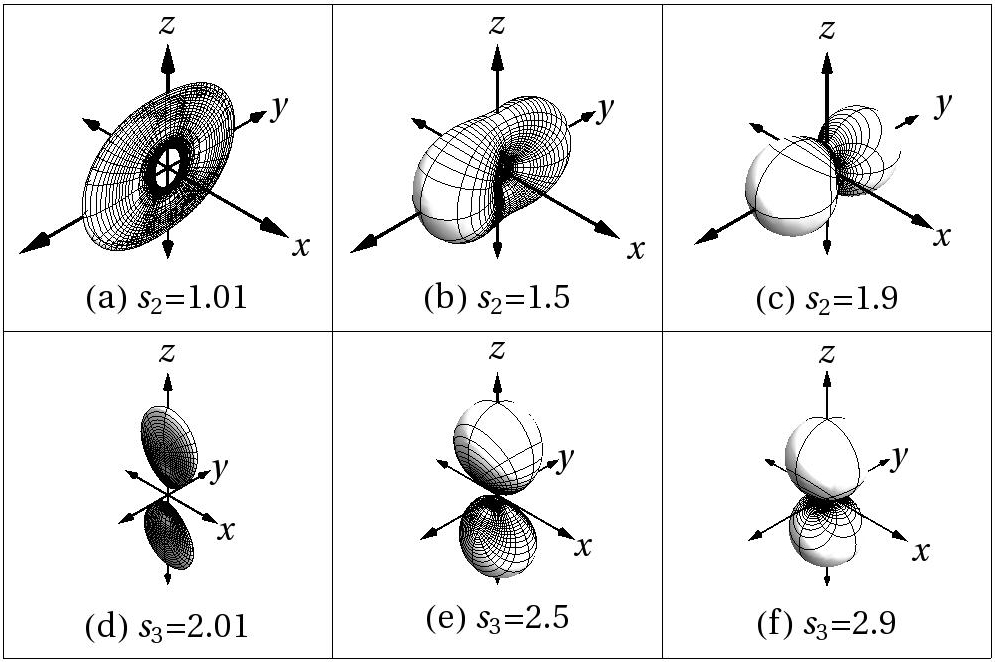}
\end{center}
\caption{Coordinate surfaces $s_{2,3}=const$ for $a_i=i$.}
\label{s23const}

\end{figure}
If $(x',y',z')\in\S_2$ then the ray $(x,y,z)=t(x',y',z')$ , $t>0$, is tangent
to the surface if and only if $(x',y',z')$ is on the surface.
If $(x',y',z')$ is in the elliptical cone
\[\frac{4x'^2}{d_2-a_1}+\frac{4y'^2}{d_2-a_2}+ \frac{4z'^2}{d_2-a_3}>0 ,\]
then the ray does not intersect the surface.
Otherwise we have two intersections $t=t_1,t_2$ and $t_1t_2=1$.
It follows from these considerations that $s_2=d_2$ describes a
connected surface of genus $1$.  The region interior to this surface is
\begin{equation}\label{8:D2}
D_2=\{(x,y,z)\in\R^3: s_2<d_2\},
\end{equation}
or, equivalently,
\[
D_2=\{(x,y,z):\frac{(x^2+y^2+z^2-1)^2}{d_2-a_0}
+\frac{4x^2}{d_2-a_1}+\frac{4y^2}{d_2-a_2}+ \frac{4z^2}{d_2-a_3}<0\}.
\]
In this section we solve the Dirichlet problem for harmonic functions
in $D_2$ by the method of separation of variables.

Let $\p=(p_0,p_1,p_2,p_3)\in\{0,1\}^4$ and $\n=(n_1,n_3)\in\N_0^2$.
Using the functions $\E{2}{i}$ introduced in Section \ref{SL2} we define
the {\it internal 5-cyclidic harmonic of the second kind}
\begin{equation}\label{8:G}
\G{2}(x,y,z)=(x^2+y^2+z^2+1)^{-1/2}\E{2}{1}(s_1)\E{2}{2}(s_2)\E{2}{3}(s_3)
\end{equation}
for $x,y,z\in \B$ with $x,y,z\ge 0$. We extend the function
\[(x^2+y^2+z^2+1)^{1/2}\G{2}(x,y,z)\]
to $\R^3$ as a function of parity $\p$.
We call a function $f$ of parity $\p=(p_0,p_1,p_2,p_3)$ if
\begin{equation}\label{8:parity}
f(\sigma_i(x,y,z))=(-1)^{p_i} f(x,y,z),\quad \mbox{\ for\ } 
i=0,1,2,3,
\end{equation}
using inversion (\ref{4:inversion}) and 
reflections (\ref{4:reflections}).

We omit the proof of the following lemma which is similar to the proof
of Lemma \ref{6:l1}.

\begin{lemma}\label{8:l1}
The function $\G{2}$ is harmonic at all points $(x,y,z)\in\R^3$
at which $s_2\ne a_2$; see (\ref{4:set4}).
\end{lemma}

Note that $s_2<d_2<a_2$ in $D_2$. Therefore, $\G{2}$ is harmonic in an
open set containing the closure of $D_2$.  Geometrically speaking,
the set $s_2=a_2$ consists of the part of the plane $y=0$ ``outside'' the
two closed curves in Figure \ref{4:fig2}.  The ring cyclide $D_2$ passes
through the $y=0$ plane inside those two closed curves.

Substituting $t_j=\Omega(s_j)$, $j=1,3$, the Hilbert space $H_2$ from
Section~\ref{SL2} transforms to the Hilbert space $\tilde{H}_2$
consisting of measurable functions
$g:(a_0,a_1)\times (a_2,a_3)\to\C$ for which
\begin{equation}\label{8:Hilbert}
\|g\|^2:= \int_{a_2}^{a_3}\int_{a_0}^{a_1}
\frac{s_3-s_1}{\omega(s_1)\omega(s_3)} |g(s_1,s_3)|^2\,ds_1\,ds_3 <\infty.
\end{equation}
By Theorem \ref{7:t3}, for $g\in \tilde{H}_2$ and fixed $\p$, we have
the Fourier expansion
\begin{equation}\label{8:Fourier}
g(s_1,s_3)\sim \sum_\n c_{\nn} \E{2}{1}(s_1)\E{2}{3}(s_3),
\end{equation}
where the Fourier coefficients are given by
\[
c_{\nn}= \int_{a_2}^{a_3}\int_{a_0}^{a_1} \frac{s_3-s_1}
{\omega(s_1)\omega(s_3)} g(s_1,s_3) \E{2}{1}(s_1)\E{2}{3}(s_3)\,ds_1\,ds_3.
\]

\begin{thm}\label{8:t1}
Consider the region $\D{2}$ defined by (\ref{8:D2}) for some
fixed $d_2\in(a_1,a_2)$.  Let $e$ be a function defined on its
boundary $\partial \D{2}$, and set
\begin{equation}\label{8:f}
 f(x,y,z):=(x^2+y^2+z^2+1)^{1/2}e(x,y,z) .
\end{equation}
Suppose that $f$ has parity $\p\in\{0,1\}^4$, and its representation
$g(s_1,s_3)$ in 5-cyclide coordinates is in $\tilde H_2$. Expand $g$ in
the series (\ref{8:Fourier}).  Then the function $u(x,y,z)$ given by
\begin{equation}
\label{8:expansion}
u(x,y,z) =\sum_\n \frac{c_{\nn}}{\E{2}{2}(d_2)} \G{2}(x,y,z)
\end{equation}
is harmonic in $\D{2}$ and assumes the values $e$ on the boundary
of $\D{2}$ in the $L^2$-sense explained below.
\end{thm}
\begin{proof}
The proof is similar to the proof of Theorem \ref{6:t1}. It uses
Theorems \ref{7:t4} and \ref{7:t5} to show that the series in
(\ref{8:expansion}) is absolutely and uniformly convergent on compact
subsets of $D_2$.  Consequently, by Lemma \ref{8:l1}, $u(x,y,z)$ is harmonic in $D_2$.
If we consider $u$ for fixed $s_2\in(a_1,d_2)$ and compute the
norm $\|u-e\|$ in the Hilbert space $\tilde H_2$, we obtain
\[
\|u-e\|^2\le \sum_\n|c_{\nn}|^2 \left(1-\frac{\E{2}{2}(s_2)}
{\E{2}{2}(d_2)}\right)^2 .\]
The right-hand side converges to $0$ as $s_2\to d_2$.
Hence $u$ assumes the boundary values $e$ in this $L^2$-sense.
\end{proof}

If $f$ is a function without parity, we write $f$ as a sum of sixteen functions
\[
f=\sum_{\p\in\{0,1\}^4} f_\p,
\]
where $f_\p$ is of parity $\p$. Then the solution of the corresponding
Dirichlet problem is given by
\begin{equation}\label{8:expansion2}
 u(x,y,z)=\sum_\nn \frac{c_{\nn}}{\E{2}{2}(d_2)} \G{2}(x,y,z),
\end{equation}
where
\begin{equation}\label{8:c2}
 c_\nn=\int_{a_2}^{a_3}\int_{a_0}^{a_1} \frac{s_3-s_1}
 {\omega(s_1)\omega(s_3)} g_\p(s_1,s_3) \E{2}{1}(s_1)\E{2}{3}(s_3)\,ds_1\,ds_3
\end{equation}
and $g_\p(s_1,s_3)$ is the representation of $f_\p$ in 5-cyclide
coordinates.  We may  also write $c_\nn$ as a surface integral
\begin{equation}\label{8:c3}
c_\nn=\frac{1}{4\omega(d_2) \E{2}{2}(d_2)} \int_{\partial \D{2}}
\frac{e}{h_2} \G{2} \, dS ,
\end{equation}
where
\[
h_2^2=\frac{1}{16}\left(\frac{(x^2+y^2+z^2-1)^2}{(d_2-a_0)^2}
+\frac{4x^2}{(d_2-a_1)^2}+\frac{4y^2}{(d_2-a_2)^2}
+\frac{4z^2}{(d_2-a_3)^2} \right) .\]

\section{Third two-parameter Sturm-Liouville problem}\label{SL3}
If we write (\ref{4:Fuchs}) on the intervals $(a_0,a_1)$ and $(a_1,a_2)$
in formally self-adjoint form, we obtain
\begin{eqnarray}
&&(\omega(s_1)w_1')'-\frac{1}{\omega(s_1)}\left(\frac3{16}s_1^2
+\lambda_1 s_1+\lambda_2\right) w_1 =0,\quad a_0<s_1<a_1,\label{9:SL1}\\
&&(\omega(s_2)w_2')'+\frac{1}{\omega(s_2)}\left(\frac3{16}s_2^2
+\lambda_1 s_2+\lambda_2\right) w_2 =0,\quad a_1<s_2<a_2.\label{9:SL2}
\end{eqnarray}
We simplify the equations by substituting $t_j=\Omega(s_j)$,
$u_j(t_j)=w_j(s_j)$, where $\Omega(s)$ is the 
elliptic integral (\ref{5:st}).
Then (\ref{9:SL1}), (\ref{9:SL2}) become
\begin{eqnarray}
 u_1''-\left(\frac3{16} \{\phi(t_1)\}^2+\lambda_1 \phi(t_1)
 +\lambda_2\right) u_1 &=&0,\quad b_0\le t_1\le b_1 ,\label{9:SL3}\\
 u_2''+\left(\frac3{16} \{\phi(t_2)\}^2+\lambda_1 \phi(t_2)
 +\lambda_2\right) u_2 &=&0,\quad b_1\le t_2\le b_2 .\label{9:SL4}
\end{eqnarray}
Of course, this system is very similar to the one considered in
Section \ref{SL1}. Therefore, we will be brief in this section.
For a given $\p=(p_0,p_1,p_2)\in\{0,1\}^3$
we consider the boundary conditions
\begin{equation}\label{9:gbc}
\begin{array}{llll}
u_1'(b_0)=0 & \mbox{\ if\ } p_0=0,&\quad  u_1(b_0)=0 &\mbox{\ if\ }
p_0=1,\\
u_1'(b_1)=u_2'(b_1)=0 &\mbox{\ if\ } p_1=0, &\quad  u_1(b_1)=u_2(b_1)=0
&\mbox{\ if\ } p_1=1,\\
u_2'(b_2)=0 &\mbox{\ if\ } p_2=0,&\quad u_2(b_2)=0 &\mbox{\ if\ } p_2=1.
\end{array}
\end{equation}
The initial conditions for $u_3$ are
\begin{equation}
\label{9:ini}
u_3(b_2)=1, u_3'(b_2)=0\, \mbox{\ if\ } p_2=0,\quad u_3(b_2)=0,
u_3'(b_2)=1\, \mbox{\ if\ } p_2=1.
\end{equation}
We denote the corresponding eigenvalues by
$(\lambda^{(3)}_{1,\nn},\lambda^{(3)}_{2,\nn})$, where $\n=(n_1,n_2)\in\N_0^2$.
The eigenfunctions will be denoted
by $\E{3}{i}(s_i)=u_{i,\n}(t_i)$, $i=1,2,3$.

Summarizing, for $i=1,2,3$, $\E{3}{i}$ is a solution
of (\ref{4:Fuchs}) on $(a_{i-1},a_i)$
with $(\lambda_1,\lambda_2)=(\lambda^{(3)}_{1,\nn},\lambda^{(3)}_{2,\nn})$.
The solution $\E{3}{1}(s_1)$ has exponent $\frac12 p_0$ at $a_0$,
exponent $\frac12 p_1$ at $a_1$, and it has $n_1$ zeros in $(a_0,a_1)$.
The solution $\E{3}{2}(s_2)$ has exponent $\frac12 p_1$ at $a_1$,
exponent $\frac12 p_2$ at $a_2$, and it has $n_2$ zeros in $(a_1,a_2)$.
The solution $\E{3}{3}(s_3)$ has exponent $\frac12 p_2$ at $a_2$, and it
has no zeros in $(a_2,a_3)$.

\section{Third Dirichlet problem}\label{Dirichlet3}

Consider the coordinate surface (\ref{4:surface}) for fixed $s=d_3\in(a_2,a_3)$.
See Figure \ref{s23const}(d,e,f) for a graphical depiction of the shape
of this surface.
If $(x',y',z')\in\S_2$ then the ray $(x,y,z)=t(x',y',z')$ , $t>0$, is tangent
to the surface if and only if $(x',y',z')$ is on the surface.
If $(x',y',z')$ is in the elliptical cone
\[\frac{4x'^2}{d_2-a_1}+\frac{4y'^2}{d_2-a_2}+ \frac{4z'^2}{d_2-a_3}<0 ,\]
then the ray intersects the surface twice at $t=t_1,t_2$ with $t_1t_2=1$.
Otherwise, there is no intersection.
Therefore, the coordinate surface $s_3=d_3$ consists of two disjoint
closed surfaces of genus $0$ separated by the plane $z=0$, and
they are mirror images of each other under the reflection $\sigma_3$.

We consider the region inside the surface $s_3=d_3$ with $z>0$
\begin{equation}\label{10:D3}
\D{3}= \{(x,y,z): z>0, s_3<d_3\},
\end{equation}
or, equivalently,
\[
\D{3}=\{(x,y,z): z>0, \frac{(x^2+y^2+z^2-1)^2}{d_3-a_0}
+\frac{4x^2}{d_3-a_1}+\frac{4y^2}{d_3-a_2}+ \frac{4z^2}{d_3-a_3}<0\}.
\]
Next, we solve the Dirichlet problem for harmonic functions in $\D{3}$ by the
method of separation of variables.

Let $\p=(p_0,p_1,p_2)\in\{0,1\}^3$ and $\n=(n_1,n_2)\in\N_0^2$. Using
the functions $\E{3}{i}$ introduced in Section \ref{SL3} we define
the {\it internal 5-cyclidic harmonic of the third kind}
\begin{equation}\label{10:G}
\G{3}(x,y,z)=(x^2+y^2+z^2+1)^{-1/2}\E{3}{1}(s_1)\E{3}{2}(s_2)\E{3}{3}(s_3)
\end{equation}
for $(x,y,z)\in \B$ with $x,y,z \ge0$.
We extend the function
\[(x^2+y^2+z^2+1)^{1/2}\G{3}(x,y,z)\]
to the half-space $\{(x,y,z): z>0\}$ as a function of parity $\p$.
We call a function $f$ of parity $\p=(p_0,p_1,p_2)$ if
\begin{equation}\label{10:parity}
f(\sigma_i(x,y,z))=(-1)^{p_i} f(x,y,z),\quad 
\mbox{\ for\ } i=0,1,2
\end{equation}
using the inversion $\sigma_0$ and the reflections $\sigma_1,\sigma_2$.
As before we have the following lemma.

\begin{lemma}\label{10:l1}
The function $\G{3}$ is harmonic on $\{(x,y,z): z>0\}$.
\end{lemma}

We have the Hilbert space
$\tilde{H}_3$ consisting of measurable functions
$g:(a_0,a_1)\times (a_1,a_2)\to\C$ for which
\begin{equation}
\label{10:Hilbert}
\|g\|^2:= \int_{a_1}^{a_2}\int_{a_0}^{a_1}
\frac{s_2-s_1}{\omega(s_1)\omega(s_1)} |g(s_1,s_2)|^2\,ds_1\,ds_2 <\infty.
\end{equation}
For $g\in \tilde{H}_3$ and fixed $\p$, we have the Fourier expansion
\begin{equation}\label{10:Fourier}
g(s_1,s_2)\sim \sum_\n c_{\nn} \E{3}{1}(s_1)\E{3}{2}(s_2),
\end{equation}
where the Fourier coefficients are given by
\begin{equation}\label{10:c1}
 c_{\nn}= \int_{a_1}^{a_2}\int_{a_0}^{a_1} \frac{s_2-s_1}
 {\omega(s_1)\omega(s_2)} g(s_1,s_2) \E{3}{1}(s_1)\E{3}{2}(s_2)\,ds_1\,ds_2 .
\end{equation}

\begin{thm}\label{10:t1}
Consider the region $\D{3}$ defined by (\ref{10:D3}) for some
fixed $d_3\in(a_2,a_3)$.
Let $e$ be a function defined on its boundary $\partial \D{3}$, and set
\begin{equation}\label{10:ef}
f(x,y,z):=(x^2+y^2+z^2+1)^{1/2}e(x,y,z) .
\end{equation}
Suppose that $f$ has parity $\p\in\{0,1\}^3$, and its representation
$g(s_1,s_2)$ in 5-cyclide coordinates
is in $\tilde H_3$. Expand $g$ in the series (\ref{10:Fourier}).
Then the function $u(x,y,z)$ given by
\begin{equation}\label{10:expansion}
 u(x,y,z) =\sum_\n \frac{c_{\nn}}{\E{3}{3}(d_3)} \G{3}(x,y,z)
\end{equation}
is harmonic in $\D{3}$ and assumes the values $e$ on the boundary
of $\D{3}$ in an $L^2$-sense.
\end{thm}

If $f$ is a function without parity, we write $f$ as a sum of eight functions
\[ f=\sum_{\p\in\{0,1\}^3} f_\p, \]
where $f_\p$ is of parity $\p$.
Then the solution of the corresponding Dirichlet problem is given by
\begin{equation}\label{10:expansion2}
 u(x,y,z)=\sum_\nn \frac{c_{\nn}}{\E{3}{3}(d_3)} \G{3}(x,y,z),
\end{equation}
where
\begin{equation}\label{10:c2}
 c_\nn=\int_{a_1}^{a_2}\int_{a_0}^{a_1} \frac{s_2-s_1}{\omega(s_1)\omega(s_2)}
 g_\p(s_1,s_2) \E{3}{1}(s_1)\E{3}{2}(s_2)\,ds_1\,ds_2
\end{equation}
and $g_\p(s_1,s_2)$ is the representation of $f_\p$ in 5-cyclide coordinates.
Alternatively, we have
\begin{equation}\label{10:c3}
c_\nn=\frac{1}{2\omega(d_3) \E{3}{3}(d_3)} \int_{\partial \D{3}}
\frac{e}{h_3} \G{3} \, dS ,
\end{equation}
where
\[ h_3^2=\frac{1}{16}\left(\frac{(x^2+y^2+z^2-1)^2}{(d_3-a_0)^2}
+\frac{4x^2}{(d_3-a_1)^2}+\frac{4y^2}{(d_3-a_2)^2}
+\frac{4z^2}{(d_3-a_3)^2} \right) .
\]

%\appendix


\section*{Acknowledgements}
This work was conducted while H.~S.~Cohl was a National Research Council
Research Postdoctoral Associate in the Applied and Computational
Mathematics Division at the National Institute of Standards and 
Technology, Gaithersburg, Maryland, U.S.A.

\section*{References}
%\bibliography{/home/hcohl/tex/refbib.bib}
%\bibliography{/Users/howiejunida/tex/refbib.bib}

\begin{thebibliography}{10}

\bibitem{AtkMing11}
F.~V. Atkinson and A.~B. Mingarelli.
\newblock {\em Multiparameter eigenvalue problems}.
\newblock Sturm-Liouville theory.
\newblock CRC Press, Boca Raton, FL, 2011.

\bibitem{Bocher}
M.~{B\^{o}cher}.
\newblock {\em Ueber die Reihenentwickelungen der Potentialtheorie}.
\newblock B.~G.~Teubner, Leipzig, 1894.

\bibitem{BoyKalMil}
C.~P. Boyer, E.~G. Kalnins, and W.~Miller, Jr.
\newblock Symmetry and separation of variables for the {H}elmholtz and
  {L}aplace equations.
\newblock {\em Nagoya Mathematical Journal}, 60:35--80, 1976.

\bibitem{BoyKalMil78}
C.~P. Boyer, E.~G. Kalnins, and W.~Miller, Jr.
\newblock {$R$}-separable coordinates for three-dimensional complex
  {R}iemannian spaces.
\newblock {\em Transactions of the American Mathematical Society},
  242:355--376, 1978.

\bibitem{Ince}
E.~L. Ince.
\newblock Ordinary differential equations.
\newblock Dover Publications, New York, 1956.

\bibitem{Kellogg}
O.~D. Kellogg.
\newblock {\em Foundations of potential theory}.
\newblock Reprint from the first edition of 1929. Die Grundlehren der
  Mathematischen Wissenschaften, Band 31. Springer-Verlag, Berlin, 1967.

\bibitem{Klein81}
F.~Klein.
\newblock Ueber {L}am\'e'sche {F}unctionen.
\newblock {\em Mathematische Annalen}, 18(2):237--246, 1881.

\bibitem{Klein91}
F.~Klein.
\newblock Ueber {N}ormirung der linearen {D}ifferentialgleichungen zweiter
  {O}rdnung.
\newblock {\em Mathematische Annalen}, 38(1):144--152, 1891.

\bibitem{Miller}
W.~Miller, Jr.
\newblock {\em Symmetry and separation of variables}.
\newblock Addison-Wesley Publishing Co., Reading, Mass.-London-Amsterdam, 1977.
\newblock With a foreword by Richard Askey, Encyclopedia of Mathematics and its
  Applications, Vol. 4.

\bibitem{MoonSpencer}
P.~Moon and D.~E. Spencer.
\newblock {\em Field theory handbook, including coordinate systems,
  differential equations and their solutions}.
\newblock Springer-Verlag, Berlin, 1961.

\bibitem{MorseFesh}
P.~M. Morse and H.~Feshbach.
\newblock {\em Methods of theoretical physics. 2 volumes}.
\newblock McGraw-Hill Book Co., Inc., New York, 1953.

\bibitem{SchmidtWolf}
D.~Schmidt and G.~Wolf.
\newblock A method of generating integral relations by the simultaneous
  separability of generalized {S}chr\"odinger equations.
\newblock {\em SIAM Journal on Mathematical Analysis}, 10(4):823--838, 1979.

\bibitem{Volkmerbook}
H.~Volkmer.
\newblock {\em Multiparameter eigenvalue problems and expansion theorems},
  volume 1356 of {\em Lecture Notes in Mathematics}.
\newblock Springer-Verlag, Berlin, 1988.

\end{thebibliography}

\bibliographystyle{amsplain}

\end{document}